\documentclass{amsart}

\usepackage[a-2u]{pdfx}
\usepackage[utf8]{inputenc}
\usepackage{lmodern}

\usepackage{amsthm}
\usepackage{amsmath}       
\usepackage{amsfonts}       
\usepackage{amssymb} 
\usepackage{bm} 
\usepackage{cleveref}

\usepackage{geometry}
\newtheorem*{remark}{Remark}
\newtheorem{theorem}{Theorem}[section]
\newtheorem{proposition}[theorem]{Proposition}
\newtheorem{lemma}[theorem]{Lemma}
\newtheorem{corollary}[theorem]{Corollary}
\newtheorem{definition}{Definition}[section]

\newcommand*{\bu}{\bm u} \newcommand*{\bv}{\bm v} \newcommand*{\bU}{\bm U}

\title[Atractor for NS with the dynamic slip in channel]{On the attractor for 2D Navier-Stokes-like system \\with the dynamic slip boundary condition in a channel}

\thanks{The author thanks the Czech Science Foundation, 
Grant Number 20-11027X, for its support.}

\author[M. Zelina]{Michael Zelina}
\address{Charles University, Faculty of Mathematics and Physics, Department of Mathematical Analysis, Sokolovsk\'{a}~83, 186~75~Prague~8, Czech~Republic}
\email{zelina@karlin.mff.cuni.cz}


\date{}

\begin{document}

\begin{abstract}
We consider a 2D infinite channel domain with an incompressible fluid satisfying the so-called dynamic slip boundary condition on the (part of the) boundary. Introducing an exhaustion by a sequence of bounded sub-domains of the whole channel we show that the unique weak solution is strong. We then construct the global attractor and find an explicit upper bound of its fractal dimension with regard to the physical parameters. This result is compatible with the analogous estimate in the case of the Dirichlet boundary condition.
\end{abstract}

\maketitle

\section{Introduction}
In this paper, we investigate the usual Navier-Stokes-like equations with viscous part of Cauchy stress $\mathbb{S}$ in the channel $\mathbb{R} \times (0, L)$. On the lower part of its boundary, we consider the following form of the dynamic slip boundary condition: 
\begin{align*}
\beta \partial_t \bu +  \alpha \bm s(\bu)  + \big[ \mathbb{S}(D \bu) \bm n \big]_{\tau} &= \beta \bm h , \\
\bu \cdot \bm n &= 0 .
\end{align*}
Here $\alpha$, $\beta$ are positive parameters, $\bm s$ is monotone non-linearity, $[\cdot]_\tau$ represents the tangential projection of a given vector, $D \bu$ denotes the symmetric gradient of $\bu$, $\bm n$ is a unit outward normal vector, and $\bm h$ is a prescribed boundary force. The most elementary case is Navier-Stokes system, i.e. $\mathbb{S} (D \bu) = 2\nu D\bu$ with $\nu > 0$ being the kinematic viscosity, together with just $\bm s (\bu) = 2\nu \bu$. This type of boundary condition is known in polymer science since the presence of the time derivative substantially helps to model situations where slip velocity might depend on the past deformation. An overview, including numerical experiments, on this and related models can be found in \cite{Hatz12}.

In recent work \cite{ABM2021}, see also \cite{EM-dis}, the existence theory for this problem was established for bounded $\Omega$. It covers a rather general class of relations between the stress tensor $\mathbb{S}$ and the shear rate $D\bu$ both of the polynomial type (Ladyzhenskaya fluid) and even implicit constitutive relations. Analogous relations between $\bm s$ and $\bu$ are considered. Let us also note that in bounded domains the existence of finite-dimensional attractors was investigated in \cite{PrZe24x2} for both 2D and 3D Ladyzhenskaya type fluid with dynamic slip boundary condition. Detailed dimension estimate was done in \cite{PrZe24} and \cite{PrPr23} for 2D and 3D cases, respectively.

We aim to extend these results into the context of unbounded domains. We focus only on $\Omega = \mathbb{R} \times (0, L)$, but the method should work in more general channel-like domains, i.e. with one bounded direction. This specific domain is interesting as due to \cite{Rosa} and \cite{IPZ2016} we have a concrete dimension estimate of the attractor in $\Omega$ with Dirichlet boundary conditions on both parts of the boundary. Hence, we would like to see how the parameters $\alpha$ and $\beta$ influence analogous estimate of the attractor dimension. Let us point out that for $\alpha \to +\infty$ our boundary condition reduces to no-slip, i.e. zero Dirichlet, boundary condition, so we expect some kind of agreement in these results.

Our paper is organized as follows. In the first section, we formulate the problem together with the needed function spaces and the notion of the solution; the exposition follows that of \cite{ABM2021} and \cite{EM-dis}. All main theorems are summarized here together with a couple of remarks. 

As there is no existence proof for such a problem in infinite domains, we thus continue in the subsequent section by proving necessary properties regarding our function spaces to show that our problem is meaningful. These are basically versions of Korn's inequality and thanks to our simple geometry we can find explicit constants in these estimates. This will be useful when finding an upper bound of the attractor dimension as the impact of parameters $\alpha$ and $\beta$ will be clear there. Compared with \cite{PrZe24}, our constants are specific and uniform in $\alpha$.

In the third section, we introduce a sequence of expanding bounded domains $Q_n$, see e.g. \cite{RRS2016} for a similar idea to construct the solution in the whole space. In these domains, a solution with $L^{\infty}(0, T; W^{1,2}(Q_n))$-regularity is then constructed just like in \cite{PrZe24}. Fortunately, all estimates are uniform and we pass to the solution in $\Omega$. Using bootstrap argument relying on \cite{PrZe25}, an unbounded analogue of \cite{AACG21}, we achieve even higher regularity in the space. 

Finally, the most important section is devoted to attractors. Just like in \cite{Rosa} we can construct the attractor and then use the method of Lyapunov exponents to find an explicit upper bound of its fractal dimension. From an extensive literature on the topic let us mention books \cite{Te97}, \cite{Robinson01}, \cite{Robinson11}, which present theory for bounded $\Omega$ (i.e. there is a compactness of embeddings), for the unbounded domains let us mention article \cite{Abergel} and its extension \cite{Rosa} for more general forces. Further, in \cite{ChepIl04} or \cite{ChepIl01} was found the same upper bound for both fractal and Hausdorff dimensions. Finally, the work \cite{IPZ2016} gives explicit majorants for constants in this upper bound. For periodic boundary condition the estimate can be substantially improved, see \cite{CFT88}, and in \cite{Liu93} is shown that it is optimal up to a logarithmic term. Finding specific constants in these estimates is related to the celebrated Lieb-Thirring inequalities, see e.g. \cite{GMT88}.

We close the introduction by briefly commenting on some possible future interests. First, it should be possible to extend our results into a non-straight channel-like domain. Next, while we covered the existence of the solution also for the 3D domain in the form $\mathbb{R}^2 \times (0, L)$, we did not establish the existence of finite-dimensional attractor here. Finally, the most interesting goal would be to estimate the attractor dimension from below. However, this is still an open question even in the simplest Dirichlet setting in a channel.

\subsection{Problem formulation}
By $\mathbb{S} (\cdot) : \mathbb{R}_{\text{sym}}^{2\times 2} \to \mathbb{R}_{\text{sym}}^{2\times 2}$ we denote the viscous part of Cauchy stress which possess the potential $U \in \mathcal{C}^2([0, +\infty))$, $U(0) = 0$, satisfying $\mathbb{S} (\mathbb{D}) = \partial_{\mathbb{D}} U (|\mathbb{D}|^2) = 2U'(|\mathbb{D}|)\mathbb{D}$. We also require the following coercivity and growth conditions (for all symmetric $2 \times 2$ matrices $\mathbb{D}$, $\mathbb{E}$)
\begin{align}
( \mathbb{S} (\mathbb{D}) - \mathbb{S} (\mathbb{E}) ) : (\mathbb{D} - \mathbb{E}) 
& \geq c_1 |\mathbb{D} - \mathbb{E}|^2 , \label{eq:Coercivity} \\
| \mathbb{S} (\mathbb{D}) | 
& \leq c_2 |\mathbb{D}| , \label{eq:Growth} \\
\partial_{\mathbb{D}} \mathbb{S} (\mathbb{D})  \mathbb{E} : \mathbb{E} 
& \geq c_3 |\mathbb{E}|^2 .  \label{eq:CoercivityDerivative}
\end{align}
Typical example is $\mathbb{S} (\mathbb{D}) = \nu (|\mathbb{D}|^2) \mathbb{D}$ for some bounded shear-dependent function $\nu (\cdot)$. Let us note that after a small modification of arguments, we can also work with $\mathbb{S}(\mathbb{D}) = \mathbb{A}(t, x, y)\mathbb{D} $, where $\mathbb{A} \in \mathcal{C}^1([0, T]; L^{\infty} (\mathbb{R}^2))$ is symmetric matrix.

On the boundary, we consider an analogue non-linearity $\bm s(\cdot) : \mathbb{R}^2 \to \mathbb{R}^2$, which is continuously differentiable, $\bm s(0) = 0$, and for all $\bu$, $\bv \in \mathbb{R}^2$ satisfies
\begin{align}
( \bm s (\bu) - \bm s (\bv) ) \cdot (\bu - \bv) 
& \geq c_1  |\bu - \bv|^2 , \label{eq:CoercivityBoundary} \\
| \bm s (\bu) | 
& \leq c_2  |\bu | , \label{eq:GrowthBoundary} \\
\nabla \bm s(\bu) \bv \cdot \bv
& \geq c_3 |\bv|^2 .  \label{eq:CoercivityDerivativeBoundary}
\end{align}
We point out that for $\mathbb{S} (\mathbb{D}) = 2\nu \mathbb{D}$ and $\bm s (\bu) = 2 \nu \bu$, where $\nu$ is positive constant, we obtain the setting of Navier-Stokes system with a linear form of the dynamic slip boundary condition. The attractor dimension for this particular problem is investigated in the fourth section.

Let us consider domain $\Omega_L = \mathbb{R} \times (0, L) \ni (x, y)$, where $L > 0$ is width of the channel. Its boundary consists of the lower part $\Gamma =  \mathbb{R} \times \{ 0 \}$ and the upper part $\Gamma_L =  \mathbb{R} \times \{ L \}$. By $\bm n = (0, \pm 1)$ we denote respective unit outward normal vectors and $[\cdot]_\tau$ represents the tangential projection of a given vector. For a prescribed time $T > 0$ and parameters $\alpha$, $\beta$, we consider the system 
\begin{align}
\partial_t \bu - \text{div}\, \mathbb{S} (D\bu)  + (\bu \cdot \nabla)\bu + \nabla \pi &=    \bm f \text{ in } (0, T) \times \Omega_L ,  \label{eq:EquationInside} \\
\beta \partial_t \bu + [ \alpha \bm s (\bu) + \mathbb{S} (D\bu) \bm n ]_\tau   &=  \beta  \bm h \text{ on } (0, T) \times \Gamma ,  \label{eq:EquationBoundary}
\end{align}
where $D\bu$ denotes the symetric gradient of $\bu$, together with
\begin{align}
\text{div}\, \bu &= 0 \text{ in } (0, T) \times \Omega_L , \label{eq:Divergence} \\
\bu &= 0  \text{ on } (0, T) \times \Gamma_L ,  \\
\bu \cdot \bm n &= 0 \text{ on } (0, T) \times ( \Gamma \cup \Gamma_L ) ,  \\
\bu(0, \cdot) &= \bu_0 (\cdot) \text{ in } \overline{ \Omega }_L .  \label{eq:Initial}
\end{align}

In our work, we will need to pay close attention to the boundary terms, we thus introduce the following subspace of smooth functions and use it to define spaces that will naturally arise in apriori estimates.
\begin{definition}
For $L > 0$, we denote
\begin{align*}
\mathcal{V}(\Omega_L) = \{ \bu \in \mathcal{C}^{\infty} (\overline{\Omega_L}); \text{ div}\, \bu = 0 \text{ in } & \Omega_L, \bu \cdot \bm n = 0 \text{ on } \Gamma , \text{supp}\, \bu \subset \Omega \cup \Gamma \text{ is compact}  \} 
\end{align*}
and introduce (for $\bu \in \mathcal{V}(\Omega_L)$ we understand $\bm g = \bu\restriction_{\Gamma}$)
\begin{align*}
V_L &= \overline{\mathcal{V}(\Omega_L)}^{V_L}, \text{ where }
|| ( \bu , \bm g) ||_{V_L}^2 = || D \bu ||_{L^2(\Omega_L)}^2  + \alpha || \bm g ||_{L^2(\Gamma)}^2,  \\
H_L &= \overline{\mathcal{V}(\Omega_L)}^{H_L}, \text{ where }
|| ( \bu , \bm g) ||_{H_L}^2 = || \bu ||_{L^2( \Omega_L)}^2 + \beta || \bm g ||_{L^2(\Gamma)}^2.
\end{align*}
For $L = 1$ we omit the lower index.
\end{definition}
Both $H_L$ and $V_L$ are Hilbert spaces with natural inner product. Further, in section 2.1, we show that $V_L$ is embedded into $W^{1,2} (\Omega_L)$. Thus, if $(\bu, \bm g) \in V_L$, then $\bm g$ is actually trace of $\bu$; hence, we will often write just $\bu \in V_L$. On the other hand, it is not the case for $(\bu, \bm g) \in H_L$, i.e. $\bu$ and $\bm g$ are not connected via trace operator in general. We will also observe that the spaces $V_L$, $H_L$, $V_L^*$ form the Gelfand triplet, which gives meaning to the duality pairing $\langle \cdot , \cdot \rangle$ in the following definition. 

\begin{definition}
Let $(\bm f , \bm h) \in L^2(0, T; V_L^*)$. By a weak solution to \eqref{eq:EquationInside}-\eqref{eq:Initial} we understand a function $\bu : (0, T) \times \Omega \rightarrow \mathbb{R}^2$, which satisfies
\begin{align*}
\bu & \in L^{\infty}(0, T; H_L) \cap L^2(0, T; V_L) , \\
\partial_t\bu & \in L^2(0, T; V_L^*) ,
\end{align*}
and for a.e. $t \in (0, T)$ and all $\bm \varphi \in V_L$ there holds the weak formulation
\begin{align}
\langle \partial_t \bu, \bm \varphi \rangle + \int\limits_{\Omega_L} \mathbb{S} (D\bu) : D \bm \varphi + \alpha \int\limits_{\Gamma} \bm s (\bu) \cdot \bm \varphi  + \int\limits_{\Omega_L}( \bu \cdot \nabla ) \bu \cdot \bm \varphi = \langle (\bm f, \bm h), \bm \varphi \rangle  \label{eq:WF}
\end{align}
together with energy equality in the form
\begin{align}
\frac{1}{2} \cdot \frac{{\rm d}}{{\rm d}t} || \bu ||_{H_L}^2  +  \int\limits_{\Omega_L} \mathbb{S} (D\bu) : D \bu  + \alpha \int\limits_{\Gamma} \bm s (\bu) \cdot \bu  = \langle (\bm f, \bm h), \bu \rangle . \label{eq:EnergyEquality}
\end{align}
\end{definition}

\subsection{Main results}
The main theorems of our paper are summarized here. First, we present the extension of regularity results from \cite{PrZe24}, where the same was done in the context of bounded domains.

\begin{theorem}\label{thm:ExistenceInfinitePlates}
Let $\bu_0 \in H_L$, $(\bm f, \bm h) \in L^2(0, T; V_L^*)$, and suppose that \eqref{eq:Coercivity}, \eqref{eq:Growth} and \eqref{eq:CoercivityBoundary}, \eqref{eq:GrowthBoundary} hold. Then there exists a unique weak solution to \eqref{eq:EquationInside}-\eqref{eq:Initial}. 

Moreover, if $(\partial_t \bm f, \partial_t \bm h) \in L^2(0, T; V_L^*)$, $(\bm f (0), \bm h (0)) \in H_L$, and both \eqref{eq:CoercivityDerivative}, \eqref{eq:CoercivityDerivativeBoundary} hold, then
\begin{align*}
\partial_t \bu & \in L_{\text{loc}}^{\infty}(0, T; H_L) \cap L_{\text{loc}}^2(0, T; V_L) , \\
\bu & \in L_{\text{loc}}^{\infty}(0, T; V_L). 
\end{align*}
Also, if $\bu_0 \in V_L \cap W^{2, 2}(\Omega_L)$, then the above regularity holds globally in time.
\end{theorem}
We note that for time-independent right-hand side $(\bm f, \bm h)$ the condition  $(\bm f (0), \bm h (0)) \in H_L$ can be dropped. Concerning the existence part of the previous theorem, instead of the conditions $U \in \mathcal{C}^2 ([0, +\infty))$ and $\bm s \in \mathcal{C}^1 (\mathbb{R}^2)$ it is enough to suppose just $\mathcal{C}^1 ([0, +\infty))$ and $\mathcal{C} (\mathbb{R}^2)$ respectively.

\begin{remark}
If we would consider $\Omega = \mathbb{R}^2 \times (0, L)$, i.e. 3D channel, then the existence part of the previous theorem also holds. The procedure is the same and the only difference is in the estimating time derivative, one gets $\partial_t \bu \in L^{\frac{4}{3}}(0, T; V_L^*)$ as in the standard setting.
\end{remark}

Using bootstrap argument we achieve the following $L^{\infty} (W^{2,p})$-regularity result.
\begin{theorem} \label{thm:MaximalRegularity}
Let $\mathbb{S}(\mathbb{D}) = 2\nu \mathbb{D}$, for some $\nu > 0$, and $\bm s \in \mathcal{C}^2(\mathbb{R}^2)$ with all its derivatives being bounded. Consider $\bu_0 \in H_L$ and $(\bm f, \bm h) \in H_L \cap ( L^p (\Omega_L) \times W^{1-\frac{1}{p}, p} (\Gamma)) $ being time-independent for certain $2 \leq p \leq 4$. Then the unique weak solution from \Cref{thm:ExistenceInfinitePlates} satisfies
\begin{align*}
\bu \in L_{\text{loc}}^{\infty} (0, T; W^{2, p} (\Omega_L))  .
\end{align*}
We also have $\nabla \pi \in L_{\text{loc}}^{\infty} (0, T; L^p (\Omega_L))$.
\end{theorem}

\begin{remark}
The previous theorem can be proven, just like in \cite{PrZe24}, also for time-dependent right-hand sides $\bm f$, $\bm h$, but one needs a rather long set of assumptions about their time derivatives.

We can also extend both results for more general boundary non-linearities. After an appropriate modification of the space $V_L$ we can, instead of \eqref{eq:CoercivityBoundary} and \eqref{eq:GrowthBoundary}, assume that
\begin{align*}
( \bm s (\bu) - \bm s (\bv) ) \cdot (\bu - \bv) 
& \geq c_1 (|\bu|^{r-2} + |\bv|^{r-2}) |\bu - \bv|^2 , \\
| \bm s (\bu) | 
& \leq c_2 |\bu|^{r-1}  
\end{align*}
hold for certain $r \geq 2$. We do not need to assume $( \bm s (\bu) - \bm s (\bv) ) \cdot (\bu - \bv) \geq c_1 (1+|\bu|^{r-2} + |\bv|^{r-2}) |\bu - \bv|^2$, as control of $L^2$-norm of $D\bu$ is enough to know that $\bu \in H_L$, see section 2. Hence, we would find that the solution $\bu$ has trace in $L^2(\Gamma) \cap L^r(\Gamma)$ for a.e. $t \in (0, T)$.

\end{remark}

The last theorem concerns the existence of the attractor $\mathcal{A}$ and its dimension; to prove it we use regularity of $\bu$ from \Cref{thm:ExistenceInfinitePlates} and a standard method of Lyapunov exponents, see e.g. \cite{Rosa}.
\begin{theorem}\label{thm:DimensionEstimate}
Let $\mathbb{S} (\mathbb{D}) = 2\nu \mathbb{D}$, $\bm s (\bu) = 2\nu \bu$ for some constant $\nu > 0$, and $(\bm f, \bm h) \in H_L$ be time-independent. Then the system \eqref{eq:EquationInside}-\eqref{eq:Initial} possess the global attractor and its fractal dimension satisfies
\begin{align}
\dim_{H_L}^f \mathcal{A}
\leq \frac{8 \kappa}{\nu^4} \left[ \frac{32L^2}{\pi^2} + \beta \min \Big\{ \frac{1}{\alpha}, 8L   \Big\}  \right]^2 || (\bm f, \bm h)||_{H_L}^2 , \label{eq:DimensionEstimate}
\end{align}
where $\kappa$ is the universal constant from \Cref{thm:Lieb-Thirring}.
\end{theorem}
This result nicely corresponds to the best-known dimension estimate for the situation when the Dirichlet boundary condition is considered also on $y = 0$ (i.e. $\alpha \to +\infty$), see e.g. \cite{IPZ2016}, in this case there holds
\begin{align*}
\dim_{L^2(\Omega_L)}^f \mathcal{A}
\leq \frac{1}{4\sqrt{3} \nu^4} \cdot \frac{L^4}{\pi^4} || \bm f||_{L^2(\Omega_L)}^2 .
\end{align*}
The constant in \eqref{eq:DimensionEstimate} behaves well also for $\alpha$ or $\beta \to 0_+$; in the latter, the estimate is the same as for $\alpha \to +\infty$. On the other hand, it explodes when $\beta \to +\infty$.

\section{Function spaces}
In the analysis of our system \eqref{eq:EquationInside}-\eqref{eq:Initial} it is enough to consider only the case of $L$ equal to 1. The general situation is then deduced from it.
\begin{lemma}\label{thm:NonDimensionalization}
Using proper scaling we can restrict ourselves to $L = 1$. Moreover, if $\nu > 0$ and $\mathbb{S} (\mathbb{D}) = 2\nu \mathbb{D}$, $\bm s (\bu) = 2\nu \bu$, then it can be assumed that $\nu = 1$.
\end{lemma}

\begin{proof}
Let us focus on the special case $\mathbb{S} (\mathbb{D}) = 2\nu \mathbb{D}$, $\bm s (\bu) = 2\nu \bu$. We introduce non-dimensional variables $t^* = \frac{t}{\tau}$ and $(x^*, y^*) = \frac{1}{L} (x, y)$, and differential operators with respect to these variables, e.g. $\nabla^* = L \nabla$. Setting $\bu^*(t^*, x^*, y^*) = \frac{1}{a}\bu(t, x, y) $ and $\pi^*(t^*, x^*, y^*) = \frac{1}{a^2} \pi(t, x, y)$, we obtain from \eqref{eq:EquationInside}-\eqref{eq:EquationBoundary} that
\begin{align*}	
\frac{a}{\tau} \partial_{t^*} \bu^* - \frac{a \nu}{L^2} \Delta^* \bu^* + \frac{a^2}{L} (\bu^* \cdot \nabla^*)\bu^* + \frac{a^2}{L} \nabla^*\pi^* &= \bm f (\tau t^*, Lx^*, Ly^*) \text{ in } \big(0, \frac{T}{\tau} \big) \times \mathbb{R} \times (0, 1) , \\			
\frac{a \beta}{\tau L} \partial_{t^*} \bu^* + \frac{2a \nu \alpha}{L} \bu^* + \frac{a \nu }{L^2} [(2D^*\bu^*)\bm n]_{\tau} &= \frac{\beta}{L} \bm h (\tau t^*, Lx^*) \text{ on } \big(0, \frac{T}{\tau} \big) \times \mathbb{R} \times \{ 0 \} .
\end{align*}
For general $\mathbb{S}$ and $\bm s$ we would introduce rescaled functions $\mathbb{S}^* (D^* \bu^*) = \frac{L}{a} \mathbb{S} (D \bu)$ and $\bm s^* (\bu^*) = \frac{1}{a}\bm s (\bu)$. We require 
\begin{align*}	
\frac{a}{\tau} = \frac{a \nu }{L^2} = \frac{a^2}{L}  
\end{align*}
to obtain $\tau = \frac{L^2}{\nu}$ and $a=\frac{\nu}{L}$ for given $L$ and $\nu>0$. Multiplying both equations by  $\frac{\tau}{a} = \frac{L^3}{\nu^2}$, we find
\begin{align*}   
\partial_{t^*} \bu^* - \Delta^*  \bu^*  + (\bu^*  \cdot \nabla^* )\bu^*  + \nabla^* \pi^*  &= \bm f^*  \text{ in } (0, T^*) \times \Omega_1 , \\          
\beta^*  \partial_{t^*} \bu^* + 2\alpha^*  \bu^* + [(2D^* \bu^* )\bm n]_{\tau} &= \beta^*  \bm h^*    \text{ on } (0, T^*) \times \Gamma ,
\end{align*}
where
\begin{align*} 
\alpha^*  = \alpha L , \quad
\beta^*  = \frac{\beta}{L}\, , \quad
T^* = \frac{\nu T}{L^2}\, , \quad
\bm f^*  = \frac{L^3}{\nu^2} \bm f(\tau t^*, Lx^*, Ly^*) , \quad
\bm h^*  = \frac{L^3}{ \nu^2} \bm h(\tau t^*, Lx^*)  . 
\end{align*}
We note that under this scaling we have the following relation 
\begin{align*}
\int\limits_{\mathbb{R}} \int\limits_0^1 | \bm f^* ( t^*, x^*, y^*)|^2 \, {\rm d}x^* \, {\rm d}y^*  
+ \beta^*\int\limits_{\Gamma}  | \bm h^* ( t^*, x^*)|^2 \, {\rm d}x^*
&= \frac{L^4}{\nu^4} \Big( \int\limits_{\mathbb{R}} \int\limits_0^L |\bm f (t, x, y)|^2 \, {\rm d}x \, {\rm d}y  +\beta \int\limits_{\Gamma} |\bm h (t, x)|^2 \, {\rm d}x \Big) ,
\end{align*} 
and thus, $|| (\bm f^*, \bm h^*) ||_H^2 = \frac{L^4}{\nu^4} || (\bm f, \bm h) ||_{H_L}^2 $.

\end{proof}

Next, we continue with several technical results regarding the spaces $V$ and $H$ (i.e. for $L = 1$). The following lemma and its corollary show that if $\bu \in V$, then we control its usual Sobolev norm, and therefore, also the trace of $\bu$; these are essentially versions of Korn's inequality. These results hold for bounded $\Omega$, see Lemma 1.11 in \cite{BMR2007}, nevertherless its proof does not work in our framework.

\begin{lemma}\label{thm:TechnicalEstimates}
Let $\bu \in \mathcal{V}(\Omega)$ and denote 
\begin{align*}
\widetilde{\bu} (x, y) = 
\begin{cases} 
\bu (x, y) & \text{ if } 0 \leq y \leq 1 \\ 
\bu (x, - y )  & \text{ if } -1 \leq y \leq 0 
\end{cases} .
\end{align*}
Then $\widetilde{\bu} (x, 1) = 0 = \widetilde{\bu} (x, -1)$, $\widetilde{\bu}$ has weak derivatives in $\mathbb{R} \times (-1, 1)$ and there holds
\begin{align}
|| D \widetilde{\bu} ||_{L^2(\mathbb{R}^2)}^2 &\leq 4|| D \bu ||_{L^2(\Omega)}^2 
\text{ and } 
|| \widetilde{\bu} ||_{L^2(\mathbb{R}^2)}^2 \leq 4|| \bu ||_{L^2(\Omega)}^2 . \label{eq:est1}
\end{align}
Moreover, we have the following relations
\begin{align}
|| \nabla \bu ||_{L^2(\Omega)}^2 &\leq 8|| D \bu ||_{L^2(\Omega)}^2 , \label{eq:FirstKornApproximation} \\
|| \bu ||_{L^2(\Omega)}^2  &\leq 8 || D \bu ||_{L^2(\Omega)}^2
\text{ and } 
|| \bu ||_{L^2(\Gamma)}^2  \leq 8 || D \bu ||_{L^2(\Omega)}^2 , \label{eq:SecondKornApproximation} \\ 
|| \widetilde{\bu} ||_{W^{1,2}(\mathbb{R}^2)}^2 &\leq 64|| D \bu ||_{L^2(\Omega)}^2 . \label{eq:est2} 
\end{align}
\end{lemma}

\begin{proof}
Let us fix $\bu \in \mathcal{V}(\Omega)$. Relations $\widetilde{\bu} (x, -1) = 0 =  \widetilde{\bu} (x, 1)$ follow from the very definition of $\mathcal{V}(\Omega)$ and we thus can trivaly extend $\widetilde{\bu}$ into whole $\mathbb{R}^2$.

\underline{Step I: weak derivatives.} Let $\bm \varphi \in \mathcal{C}^\infty_c (\mathbb{R}^2)$, we have
\begin{align*}
\int\limits_{\mathbb{R}} \int\limits_{-1}^1 \widetilde{\bu} (x, y) \bm \varphi_y (x, y) \, {\rm d}y \, {\rm d}x
&= \int\limits_{\mathbb{R}} \int\limits_0^{1}  \bu (x, y)  ( \bm \varphi_y (x, y) + \bm \varphi_y (x,  - y)) (1 \pm \eta_{\varepsilon} (y)) \, {\rm d}y \, {\rm d}x , 
\end{align*}
where $ \eta_{\varepsilon} (y) = \begin{cases} 1 & \text{ if } |y| > 2\varepsilon \\ 0 & \text{ if } |y| < \varepsilon \end{cases} $ is a smooth function with $|\eta_{\varepsilon}' | \leq \frac{c}{\varepsilon}$; note that $\bm \varphi (x, y) - \bm \varphi (x, - y) \notin \mathcal{C}^{\infty}_c (\Omega)$. After integration by parts and letting $\varepsilon \to 0_+$, we find that 
\begin{align*}
\int\limits_{\mathbb{R}} \int\limits_{-1}^{1} \widetilde{\bu} (x, y) \bm \varphi_y (x, y) \, {\rm d}y \, {\rm d}x
&= -\int\limits_{\mathbb{R}} \int\limits_{-1}^{1} \widetilde{\bu_y} (x, y)  \bm \varphi (x, y)  \, {\rm d}y \, {\rm d}x .
\end{align*}
Analogue, but simpler calculation works also for the derivative with respect to $x$. Therefore, $\widetilde{\bu}$ is a reasonable extension and \eqref{eq:est1} indeed holds.

\underline{Step II: relation between $\nabla \bu$ and $D\bu$.} As $\widetilde{\bu} $ is compactly supported in $\mathbb{R}\times (-1, 1)$, we can molify it to find $\widetilde{\bu}_{\varepsilon} \in \mathcal{C}^{\infty}_c (\mathbb{R} \times (-1, 1))$ and  $ || D \widetilde{\bu}_{\varepsilon} ||_{L^2( \mathbb{R} \times (-1, 1) )} \leq || D \widetilde{\bu} ||_{L^2( \mathbb{R} \times (-1, 1) )} \leq 2|| D \bu ||_{L^2(\Omega)}$. Due to differentiability and compact support of $\widetilde{\bu}_{\varepsilon}$ we can now show that
\begin{align*}
\int\limits_{\mathbb{R}}\int\limits_{-1}^{1} |\nabla \widetilde{\bu}_{\varepsilon}|^2
\leq 2 \int\limits_{\mathbb{R}}\int\limits_{-1}^{1}|D \widetilde{\bu}_{\varepsilon} |^2   ,
\end{align*}
and using weak lower semicontinuity of norms we obtain \eqref{eq:FirstKornApproximation}.

\underline{Step III: estimate of $\bu$ in $L^2$.} As $\Omega$ is bounded in $y$-direction and $\bu (x, 1) = 0$ for all $ x \in \mathbb{R}$, we can write $|\bu (x, y)|^2 = \Big| \int\limits_y^1 \frac{\partial \bu}{\partial y}  (x, s) \, {\rm d}s  \Big|^2$ and using \eqref{eq:FirstKornApproximation} we get 
\begin{align*}
|| \bu ||_{L^2(\Omega)}^2 
&\leq  \int\limits_{\mathbb{R}}  \int\limits_0^1 \int\limits_0^1 \Big| \frac{\partial \bu}{\partial y} (x, s) \Big|^2 \, {\rm d}s \, {\rm d}y  \, {\rm d}x 
\leq  \int\limits_{\Omega} |\nabla \bu|^2 
\leq 8 \int\limits_{\Omega} |D \bu|^2 .
\end{align*}
Similarly, 
\begin{align*}
|| \bu ||_{L^2(\Gamma)}^2 
= \int\limits_{\mathbb{R}} \Big| \int\limits_0^1 \frac{\partial \bu(x, s)}{\partial y} {\rm d}s \Big|^2 {\rm d}x
\leq  \int\limits_{\Omega} |\nabla \bu|^2 
\leq  8 \int\limits_{\Omega} |D \bu|^2 ,
\end{align*}
and \eqref{eq:SecondKornApproximation} holds. Finally, using \eqref{eq:est1}, its obvious analogue for $\nabla \bu$, \eqref{eq:SecondKornApproximation}, and \eqref{eq:FirstKornApproximation} together, we obtain
\begin{align*}
|| \widetilde{\bu} ||_{W^{1,2}( \mathbb{R} \times (-1, 1)) }^2 \leq 4|| \bu ||_{L^2(\Omega)}^2  + 4|| \nabla \bu ||_{L^2(\Omega)}^2  \leq 64 || D \bu ||_{L^2(\Omega)}^2 , 
\end{align*}
from which \eqref{eq:est2} follows.
\end{proof}

\begin{corollary}\label{thm:BadKorn}
Let $\bu \in V$. Then there exists a linear extension $E\bu$ of $\bu$ into $\mathbb{R}^2$ such that 
\begin{align*}
|| \nabla (E \bu ) ||_{L^2(\mathbb{R}^2)}^2 &\leq 8|| D\bu ||_{L^2(\Omega)}^2  , \,
|| E \bu ||_{L^2(\mathbb{R}^2)}^2 \leq 4|| \bu ||_{L^2(\Omega)}^2  , \text{ and }
|| E \bu ||_{W^{1,2}(\mathbb{R}^2)}^2 \leq 64|| D \bu ||_{L^2(\Omega)}^2  ,
\end{align*}
and $E\bu = 0$ a.e. outside of $\mathbb{R} \times (-1, 1)$ and has zero trace on $y = \pm 1$. There also hold the following inequalities
\begin{align}
|| \nabla \bu ||_{L^2(\Omega)}^2 &\leq 8 || D \bu ||_{L^2(\Omega)}^2 , \label{eq:FirstKorn} \\
|| \bu ||_{L^2(\Omega)}^2 &\leq 8|| D \bu ||_{L^2(\Omega)}^2 
\text{ and }
 || \bu ||_{L^2(\Gamma)}^2 \leq 8|| D \bu ||_{L^2(\Omega)}^2 . \label{eq:SecondKornWorse}
\end{align}
\end{corollary}

\begin{proof}
For $\bu \in V$ we find a sequence $\{ \bu_n \}_{n \in \mathbb{N}} \subset \mathcal{V}(\Omega)$ such that $|| \bu - \bu_n ||_{V} \leq \frac{1}{n}$. Now, \Cref{thm:TechnicalEstimates} gives us functions $\widetilde{\bu}_n$, which satisfy
\begin{align*}
|| \widetilde{\bu}_n - \widetilde{\bu}_m||_{W^{1,2}(\mathbb{R}^2)} ^2
&= || \widetilde{\bu_n - \bu_m}||_{W^{1,2}(\mathbb{R}^2)} ^2
\leq 64|| D(\bu_n - \bu_m)||_{L^2(\Omega)}^2 \\
&= 64|| D\bu_n - D\bu_m ||_{L^2(\Omega)}^2
\leq 64|| D\bu_n - D\bu_m ||_{V}^2.
\end{align*}
So, $\{ \widetilde{\bu}_n \}_{n \in \mathbb{N}}$ is Cauchy sequence, and because $W^{1,2}(\mathbb{R}^2)$ is a Banach space, $\{ \widetilde{\bu}_n \}_{n \in \mathbb{N}}$ converges to certain $E\bu$. Clearly, $E\bu = 0$ a.e. outside of $\mathbb{R} \times (-1, 1)$ and because it is a Sobolev function, we know that trace operator is well-defined, and so, $\text{tr}\, E\bu = 0$ on $y = -1$ and $y = 1$. First three inequalities in the statement then follow from Step II in the proof of \Cref{thm:TechnicalEstimates}, \eqref{eq:est1} and \eqref{eq:est2}. Finally, the estimates \eqref{eq:FirstKorn} and \eqref{eq:SecondKornWorse} follow from \eqref{eq:FirstKornApproximation} and \eqref{eq:SecondKornApproximation}, respectively, which are valid for the approximations $\bu_n$.
\end{proof}

\begin{remark}
If $\bu \in V \cap W^{2, p}(\Omega)$ for some $p \in (1, +\infty)$, then there is a linear extension $E\bu$ of $\bu$ satisfying
\begin{align*}
|| E \bu ||_{W^{1,p}(\mathbb{R}^2)} + || E \bu ||_{W^{2,p}(\mathbb{R}^2)}  \leq C ( || \bu ||_{W^{1,p}(\Omega)} + || \bu ||_{W^{2,p}(\Omega)} ) .
\end{align*}
\end{remark}

\begin{proof}
For $\bu \in \mathcal{V}(\Omega)$ one can define
\begin{align*}
\widetilde{\bu} (x, y) = 
\begin{cases} \bu (x, y) & \text{ if } 0 \leq y \leq 1 \\ 
5\bu (x, - y ) -20\bu  \big(x, -\frac{1}{2} y  \big) + 16\bu  \big(x,  -\frac{1}{4} y  \big)  & \text{ if } -1 \leq y \leq 0 \end{cases} .
\end{align*}
It is easy to check that derivatives from the left and from the right agree on $y = 0$, which means that $\widetilde{\bu} \in \mathcal{C}^2 (\mathbb{R} \times [-1, 1])$. However, value of $\widetilde{\bu} $ on $y = -1$ is not zero. To deal with it, one can introduce a suitable smooth cut-off function. It gives us an extension on $\mathbb{R} \times [-1, 1]$, which also has a zero trace. To extend it into a whole space we then use standard extension results. 
\end{proof}

The following version of the celebrated Ladyzhenskaya's inequality will be needed.
\begin{corollary}
There exists $C > 0 $ such that for any $\bu \in V$ we have the estimate
\begin{align}
|| \bu ||_{L^4(\Omega)}^2 &\leq C || \bu ||_{L^2(\Omega)} || \bu ||_{V} . \label{eq:Ladyzhenskaya}
\end{align}
\end{corollary}

\begin{proof}
Due to \Cref{thm:BadKorn} and standard Ladyzhenskaya's inequality we find that for any $\bu \in V$  holds 
\begin{align*}
|| \bu ||_{L^4(\Omega)}^2
\leq ||  E\bu ||_{L^4(\mathbb{R}^2)}^2
\leq  \sqrt{2} ||  E\bu ||_{L^2(\mathbb{R}^2)} || E\bu ||_{W^{1,2}(\mathbb{R}^2)}
\leq  C ||  \bu ||_{L^2(\Omega)} || \bu ||_{V} .
\end{align*}
\end{proof}

The estimates in \Cref{thm:BadKorn} can be slightly improved. The best result would follow by finding
\begin{align*}
\lambda = \min_{0 \neq \bu \in V} \frac{|| \bu ||_{V} }{ || \bu ||_{W^{1,2}(\Omega)} } \, ,
\end{align*}
but this task would be rather complicated; primarily due to divergence-free condition. We will thus relax the conditions from $V$ and due to \eqref{eq:FirstKorn} we will work only with gradient instead of its symmetrical part. 

\begin{lemma}\label{thm:Korn}
Let $\bu \in V$. Then the following estimate holds
\begin{align}
|| \bu ||_{L^2(\Omega)}^2 &\leq  \frac{8}{\lambda^2(\alpha)}  || \bu ||_{V}^2 , \label{eq:SecondKorn} 
\end{align}
where $\lambda(\cdot)$ is an implicitly given function satisfying $\frac{1}{4}\pi^2 \leq \lambda^2(\alpha) \leq \pi^2$.
\end{lemma}

\begin{proof}
Let us consider the space
$$ \mathfrak{X} = \{  \bu \in \mathcal{C}^{\infty} (\overline{\Omega}); \bu \not\equiv 0,  \bu \cdot \bm n = 0  \text{ on } \Gamma , \text{supp}\, \bu \subset \Omega \cup \Gamma \text{ is compact}  \}  $$
and its closure $\overline{\mathfrak{X}}^{|| \cdot ||_V}$. We will show that the smallest positive eigenvalue $\lambda^2$ of the system (where $\bu \in \overline{\mathfrak{X}}^{|| \cdot ||_V}$)
\begin{align}
\begin{split}
\Delta \bu + \lambda^2 \bu &= 0 \text{ in } \Omega , \\
8\alpha u_1 - \partial_y u_1 &= 0 \text{ on } \Gamma , \\
u_2 &= 0  \text{ on } \Gamma , \\
\bu &= 0  \text{ on } \Gamma_1 , \label{eq:EulerLagrange}
\end{split}
\end{align}
satisfies $\frac{1}{4}\pi^2 \leq \lambda^2 \leq \pi^2$. This system represents Euler-Lagrange equations corresponding to the variational problem
\begin{align*}
\lambda^2 = \min_{\bu \in \overline{\mathfrak{X}}^{|| \cdot ||_V}} \frac{|| \nabla \bu ||_{L^2(\Omega)}^2 + 8\alpha || \bu ||_{L^2(\Gamma)}^2}{ || \bu ||_{L^2(\Omega)}^2 } \, .
\end{align*}
When we find it, then there holds $|| \bu ||_{L^2(\Omega)}^2  \leq \frac{1}{\lambda^2} ( || \nabla \bu ||_{L^2(\Omega)}^2 + 8\alpha || \bu ||_{L^2(\Gamma)}^2 )  $ and using  \eqref{eq:FirstKorn} we obtain \eqref{eq:SecondKorn}. 

It remains to find $\lambda^2$. It is simple to check that it is indeed a positive real number. Equations in \eqref{eq:EulerLagrange} are essentially two separated systems for $u_1$ and $u_2$. We start with the simpler one, i.e.
\begin{align*}
\Delta u_2 + \lambda^2 u_2 &= 0 \text{ in } \Omega , \\
u_2 &= 0  \text{ on } \Gamma \cup \Gamma_1 . 
\end{align*}
We first look for the (non-trivial) solution in the form of separted variables, in other words, let $u_2(x, y) = X(x) Y(y)$, we need to solve
\begin{align*}
Y''(y) &= - \mu^2 Y(y) , Y(0) = Y(1) = 0 , \\
X''(x) &= (\mu^2 - \lambda^2 ) X(x)  .
\end{align*}
The first equation gives us directly that $\mu_k = \pi k$ and
$$ Y_k (y) = a_k \sin \mu_k y, k \in \mathbb{N}, a_k \in \mathbb{R} . $$
We see that $\mu_k^2 - \lambda^2 < 0$, otherwise $X$ would be unbounded and we know that $|u_2(x, y)| \to 0$ as $|x| \to +\infty$. It is possible to then express $X(x)$ and $u_2$, but we do not find more information about $\lambda^2$. So, $\lambda \geq \pi^2$.

The same procedure applies for $u_1$, we deal with
\begin{align*}
Y''(y) &= - \mu^2 Y(y)  ,\\
8\alpha Y(0) - Y'(0) &= 0 ,\\
Y(1) &= 0 .
\end{align*}
The case $\mu^2 < 0$ gives $Y(y) = a (e^{\mu y} - e^{2\mu}e^{-\mu y})$ which leads to the trivial solution only, just like the case $\mu^2 = 0$. So, let $\mu^2 > 0$, then $Y(y) = a \big( \sin \mu y + \frac{\mu}{8\alpha} \cos \mu y \big)$, $\mu > 0$, and 
\begin{align*}
\mu \cos \mu  + 8\alpha \sin \mu  = 0 
\end{align*}
needs to be satisfied, which gives us solutions $\mu = \mu(\alpha)$. For $\alpha \in \{ 0, 0.1, 1, 10, \infty \}$ the smallest $\mu$ is $\{ \frac{\pi}{2}, 1.956, 2.804, 3.103, \pi \}$. From the equation for $X$ we find again $\mu^2 \leq \lambda^2$, so $\lambda^2 \geq \frac{1}{4}\pi^2$. 

For the minimizer $\bu$, $u_2 \equiv 0$ is admissible, and therefore, $\frac{1}{4}\pi^2 \leq \lambda^2(\alpha) \leq \pi^2$. Also, it is not difficult to observe that $\lambda (\cdot)$ is a continuous and increasing function of $\alpha$.

\end{proof}

The main result of this section follows. We note that the specific dependence on $\alpha$, $\beta$ and $L$ between $H_L$ and $V_L$-norms will be important in the fourth section.

\begin{corollary}\label{thm:CrucialEstimates}
Let $\bu \in V_L$, then
\begin{align}
|| \bu ||_{W^{1,2}(\Omega_L)}^2 &\leq 8 \Big( 1 + \frac{4L^2}{\pi^2} \Big) || \bu ||_{V_L}^2 , \label{eq:Korn} \\
|| \bu ||_{H_L}^2 & \leq  \Lambda  ||\bu||_{V_L}^2  ,  \label{eq:EquivalenceNorms}
\end{align}
where $\Lambda =  \frac{32}{\pi^2}L^2 + \beta \min \big\{  \frac{1}{\alpha},  8L \big\}$.
\end{corollary}

\begin{proof}
For $L = 1$, \Cref{thm:Korn} says $|| \bu ||_{L^2(\Omega)}^2 \leq  \frac{32}{\pi^2}  || \bu ||_{V}^2$, and in the spirit of \Cref{thm:NonDimensionalization} we see that
\begin{align}
|| \bu ||_{L^2(\Omega_L)}^2 &\leq  C_1 || \bu ||_{V_L}^2  \label{eq:est5}
\end{align}
with $C_1 = \frac{32L^2}{\pi^2}$. Similarly, from \eqref{eq:FirstKorn} we get
\begin{align*}
|| \nabla \bu ||_{L^2(\Omega_L)}^2 &\leq  8  || D \bu ||_{L^2(\Omega_L)}^2 .
\end{align*}
Hence, adding these inequalities, the estimate \eqref{eq:Korn} follows. Now, consider $\bu \in V_L$ and $\theta = \frac{\beta}{\beta + \alpha C_1} \in (0, 1)$. Using \eqref{eq:est5} we then obtain
\begin{align*}
|| \bu ||_{V_L}^2 
= \theta || \bu ||_{V_L}^2 + (1 - \theta) || \bu ||_{V_L}^2 
\geq \frac{\alpha \theta}{\beta} \beta || \bu ||_{L^2(\Gamma)}^2 + \frac{1-\theta}{C_1} || \bu ||_{L^2(\Omega_L)}^2
= \frac{\alpha}{\beta + \alpha C_1} ||\bu||_{H_L}^2 . 
\end{align*}
Of course, this estimate favors $\alpha \gg 1$. If $\alpha \ll 1$, we might proceed differently. From \eqref{eq:SecondKornWorse} we deduce that 
\begin{align*}
|| \bu ||_{L^2(\Gamma)}^2 &\leq  C_2 || D \bu ||_{L^2(\Omega_L)}^2 ,
\end{align*}
where $C_2 = 8L$. Considering now $\theta = \frac{\beta C_2}{C_1 + \beta C_2}$ we apply this inequality together with \eqref{eq:est5} to find
\begin{align*}
|| \bu ||_{V_L}^2 
&\geq \frac{\theta}{\beta C_2} \beta || \bu ||_{L^2(\Gamma)}^2 + \frac{1-\theta}{C_1} || \bu ||_{L^2(\Omega_L)}^2 = \frac{1}{C_1 + \beta C_2} ||\bu||_{H_L}^2 .
\end{align*}
Hence $\max \Big\{\frac{\alpha}{\beta + \alpha C_1} ,  \frac{1}{C_1 + \beta C_2 } \Big\} ||\bu||_{H_L}^2 \leq ||\bu||_{V_L}^2$ and the result follows.
\end{proof}
Note that when $\beta \to +\infty$, then $\Lambda $ tends to $+\infty$. This will later result in blow-up in the constant in the estimate of the dimension of the attractor for large $\beta$; see \eqref{eq:DimensionEstimate}. On the other hand, if $\alpha$ or $\beta$ are close to zero, then $\Lambda$ is well-behaved.

The above result shows that $V_L \hookrightarrow H_L$. It is now a simple observation that $V_L \hookrightarrow H_L \hookrightarrow V_L^*$ and both embeddings are continuous and dense. Further, duality pairing between $V_L$ and $V_L^*$ can be now defined in a standard way as a continuous extension of the inner product on $H_L$.

\begin{remark}
All results in this section could be extended to $d$-dimensional channel $\mathbb{R}^{d-1} \times (0, L) $ instead of $\Omega$.
\end{remark}

\section{Existence of strong solution}
Our goal now is to show the existence of the unique solution $\bu$ and its higher regularity. In this section, we focus solely on $L = 1$. The proof relies on the approximation of $\Omega$ by a sequence of bounded domains $\{ Q_n\}_{n \in \mathbb{N}}$ in the spirit of Theorem 4.10 from \cite{RRS2016}. For bounded $\Omega$, the regularity result is known due to \cite{PrZe24}. Nevertherless, we can't just invoke that result, as it is unclear if the constants would be independent of $n$. We start with $L^\infty (W^{1, 2})$-regularity, which is based on the Galerkin approximation. The procedure is the same as in a couple of relevant theorems from \cite{PrZe24}, however, we will do that in detail here because of the parameter $n$ and also due to the non-linear terms.

\begin{proof}[Proof of \Cref{thm:ExistenceInfinitePlates}] The proof is divided into six parts. We point out that we will not relabel the original sequence when selecting a subsequence, and $C$ stands for a generic constant whose value may change line to line.

\underline{Step I: reduction to bounded domains.} For $n \in \mathbb{N}$ we introduce the sets
\begin{align*}
Q_n &= (-n, n) \times (0, 1) , \, \Gamma_n = (-n, n) \times \{ 0 \} 
\end{align*}
and the space
\begin{align*}
\mathcal{V}(Q_n) = \{ \bu  \in \mathcal{C}^{\infty} (\overline{\Omega}); \text{ div}\, \bu = 0 \text{ in }  Q_n, \bu \cdot \bm n = 0 \text{ on } \Gamma_n, \text{supp}\, \bu \subset (-n, n) \times [0, 1) \} .
\end{align*}
Given our data, we clearly can construct the sequence $\{ \bu_{0, n} \}_{n \in \mathbb{N}} \subset \mathcal{V}(Q_n)$ such that
\begin{align}
\bu_{0, n} \to \bu_0 &\text{ in } H  \text{ and }  || \bu_{0, n} ||_{H(Q_n)} \leq || \bu_0 ||_H  . \label{eq:MonotoneNorms}
\end{align}
Here, $V(Q_n)$ is the closure of $\mathcal{V}(Q_n)$ in the norm $|| (\bu, \bm g) ||_{V(Q_n)}^2 = || D \bu ||_{L^2(Q_n)}^2 + \alpha || \bm g ||_{L^2(\Gamma_n)}^2 $ and analogously is defined $H(Q_n)$. In the spirit of section 2, one observes that these spaces are Hilbert spaces, analogue of \Cref{thm:CrucialEstimates} holds and there is the Gelfand triplet $V(Q_n) \hookrightarrow H(Q_n)  \hookrightarrow V^*(Q_n) = (V(Q_n))^*$. Notice that $V(Q_n) \hookrightarrow V$, hence $V^* \hookrightarrow V^*(Q_n)$ and we have 
\begin{align}
|| (\bm f, \bm h) ||_{L^2(0, T; V^*(Q_n))} \leq  || (\bm f, \bm h) ||_{L^2(0, T; V^*)}   . \label{eq:MonotoneNorms2}
\end{align}

Let us note that any function from $\mathcal{V}(\Omega)$ can be approximated (in $V$ or $H$ norm) by a function from  $\mathcal{V}(Q_n)$ for $n \in \mathbb{N}$ big enough. Hence, if $\bu \in V$ then there exists a sequence $\{ \bu_n \}_{n \in \mathbb{N}}$ of functions from $V(Q_n)$ such that $\bu_n \to \bu$ in $V$.

\underline{Step II: existence of solutions in $Q_n$.} Now, let $n \in \mathbb{N}$ be fixed. We are dealing with 
\begin{align*}
\partial_t \bu_n - \text{div}\, \mathbb{S} (D\bu_n)  + (\bu_n \cdot \nabla)\bu_n + \nabla \pi_n &=    \bm f \text{ in } (0, T) \times Q_n ,   \\
\beta \partial_t \bu_n + [ \alpha \bm s (\bu_n) + \mathbb{S} (D\bu_n) \bm n ]_\tau   &=  \beta  \bm h \text{ on } (0, T) \times \Gamma_n ,  \\
\bu_n &= 0 \text{ on } (0, T) \times \partial Q_n \setminus \Gamma_n ,
\end{align*}
together with divergence-free condition, impermeability and the initial condition $\bu_{0, n}$. To solve it, we modify a standard existence proof from \cite{PrZe24} or \cite{ABM2021} (also for 3D case). We need the following lemma. 

\begin{proposition}\label{thm:Basis}
There exists the sequence $\{ \bm \omega_k \}_{k\in \mathbb{N}}$ which is a basis in both $V(Q_n)$ and $H(Q_n)$, it is orthogonal in $V(Q_n)$ and orthonormal in $H(Q_n)$. Further, there is a non-decreasing sequence $\{\mu_k\}_{k\in \mathbb{N}}$ with $\lim_{k \rightarrow +\infty} \mu_k = +\infty$. For every $k \in \mathbb{N}$, the function $\bm \omega_k$ solves 
\begin{align*}
(\bm \omega_k, \bm \varphi)_{V(Q_n)} = \mu_k (\bm \omega_k, \bm \varphi)_{H(Q_n)}  
\end{align*}
for all $\bm \varphi \in V(Q_n) $. Moreover, for the projection $P^M$ of $V(Q_n)$ to the linear hull of $\{ \bm \omega_k \}_{k = 1}^M$ defined by the formula $ P^M \bu = \sum_{k = 1}^M (\bu, \bm \omega_k)_{H(Q_n)} \bm \omega_k$, it holds that for any $\bu \in V(Q_n)$
\begin{align*}
||P^M \bu ||_{H(Q_n)} & \leq || \bu ||_{H(Q_n)}, \, ||P^M \bu ||_{V(Q_n)}  \leq || \bu ||_{V(Q_n)}, \text{ and } 
P^M \bu  \rightarrow \bu \text{ in } V(Q_n) \text{ as } M \rightarrow +\infty .
\end{align*}
\end{proposition}

\begin{proof}
It is a straightforward analogue of Lemma A.1. in \cite{ABM2021} or Lemma 3.1. in \cite{EM-dis}.
\end{proof}

Because of it, we introduce the Galerkin $m$-th level approximation
\begin{align*}
\bu_n^m (t, x, y) = \sum_{k = 1}^m c_{n, k}^m (t) \bm \omega_k(x, y)
\end{align*}
which satisfies (for $k = 1, \dots , m$)
\begin{align}
( \partial_t \bu_n^m, \bm \omega_k )_{H(Q_n)} + \int\limits_{Q_n} \mathbb{S}( D \bu_n^m) : D \bm \omega_k + \alpha \int\limits_{\Gamma_n} \bm s(\bu_n^m) \cdot \bm \omega_k  + \int\limits_{Q_n} (\bu_n^m\cdot \nabla) \bu_n^m\cdot \bm \omega_k &= \langle (\bm f, \bm h), \bm \omega_k \rangle \label{eq:Galerkin}
\end{align}
and
\begin{align*}
c_{n,k}^m(0) &= ( \bu_{0, n} , \bm \omega_k)_{H(Q_n)} .
\end{align*}
We multiply \eqref{eq:Galerkin} by $c_{n,k}^m$, sum over $k$'s and integrate over time inteval $(0, t)$ to find
\begin{align}
\frac{1}{2}|| \bu_n^m(t) ||_{H(Q_n)}^2 + \int\limits_0^t \Big( \int\limits_{Q_n} \mathbb{S}( D \bu_n^m) : D \bu_n^m  + \alpha \int\limits_{\Gamma_n} \bm s (\bu_n^m) \cdot \bu_n^m  \Big)  \leq \frac{1}{2}|| \bu_{0, n} ||_{H(Q_n)}^2 +  \int\limits_0^t \langle (\bm f, \bm h), \bu_n^m\rangle . \label{eq:EnergyGalerkin}
\end{align}
Due to \eqref{eq:Coercivity}, \eqref{eq:CoercivityBoundary}, \eqref{eq:MonotoneNorms}, Gr{\"o}nwall's inequality  and the duality argument, we achieve
\begin{align}
|| \bu_n^m||_{L^2(0, T; V(Q_n))} + || \bu_n^m||_{L^{\infty}(0, T; H(Q_n))} + || \partial_t \bu_n^m||_{L^2(0, T; V^*(Q_n))}  \leq C, \label{eq:UniformEstimates}
\end{align}
where $C$ is uniform in both $m$ and $n$. Therefore, $\bu_n^m$ converges (in $m$) to weak limit denoted by $\bu_n$. Thanks to (analogue of) \Cref{thm:CrucialEstimates} we control also $L^2(W^{1,2})$-norm, and due to Aubin-Lions theorem we get strong convergence in $L^2(0, T; L^2(Q_n))$. Due to \eqref{eq:Growth} and \eqref{eq:GrowthBoundary} we obtain the following weak convergences
\begin{align*}
\mathbb{S}(D\bu_n^m) \rightharpoonup \overline{\mathbb{S}(D\bu_n)} \text{ in } L^2(Q_n) \text{ and } \bm s (\bu_n^m) \rightharpoonup \overline{\bm s (\bu_n)} \text{ in } L^2(\Gamma_n) \text{ as } m \to + \infty .
\end{align*}

Next, we consider a test function $\bm \varphi \in V(Q_n)$ and $\psi \in \mathcal{C}_c^{\infty}(0, T)$ and multiply $k$-th equation of \eqref{eq:Galerkin} by $\psi (\bm \varphi, \bm \omega_k)_{H(Q_n)}$ and sum over $k = 1, \dots , M$, $M \leq m$. We then integrate the result over the time interval and use our convergences to pass with $m \to +\infty$. Next, we employ convergence $P^M \bm \varphi \to \bm \varphi$ in $V(Q_n)$ from \Cref{thm:Basis} to show that for a.e. $t \in (0, T)$ and all $\bm \varphi \in V(Q_n)$ there holds
\begin{align*}
\langle \partial_t \bu_n, \bm \varphi \rangle_{V(Q_n)} + \int\limits_{Q_n} \overline{\mathbb{S}(D\bu_n)} : D \bm \varphi + \alpha \int\limits_{\Gamma_n} \overline{\bm s (\bu_n)} \cdot \bm \varphi + \int\limits_{Q_n}( \bu_n \cdot \nabla ) \bu_n \cdot \bm \varphi = \langle (\bm f, \bm h), \bm \varphi \rangle _{V(Q_n)} . 
\end{align*}

Thanks to $\bu \in L^2 (0, T; V)$ and $\partial_t \bu \in L^2 (0, T; V^*)$, we see that $\bu_n \in \mathcal{C}([0, T]; H(Q_n))$, hence the initial data attainment is clear. It remains to identify the limits of non-linear terms, i.e. we wish to show $\overline{\mathbb{S}(D\bu_n)} = \mathbb{S}(D\bu_n)$ and $ \overline{\bm s (\bu_n)} =  \bm s (\bu_n)$. As the analogous conclusion will also be needed for $n \to + \infty$, we postpone it to Step V, where this more general case will be treated. Thus, $\bu_n$ solves our system in $Q_n$, i.e. for a.e. $t \in (0, T)$ and all $\bm \varphi \in V(Q_n)$ we have
\begin{align}
\langle \partial_t \bu_n, \bm \varphi \rangle_{V(Q_n)} + \int\limits_{Q_n} \mathbb{S}(D\bu_n) : D \bm \varphi + \alpha \int\limits_{\Gamma_n} \bm s (\bu_n) \cdot \bm \varphi + \int\limits_{Q_n}( \bu_n \cdot \nabla ) \bu_n \cdot \bm \varphi = \langle (\bm f, \bm h), \bm \varphi \rangle _{V(Q_n)} . \label{eq:WFApprox}
\end{align}

\underline{Step III: existence of solution in $\Omega$.} Functions $\bu_n$ have zero trace on $\{\pm n \} \times [0, 1]$, so we can extend each $\bu_n$ by zero to the whole $\Omega$ and we immediately get the following estimate
\begin{align*}
|| \bu_n ||_{L^2(0, T; V)} + || \bu_n ||_{L^{\infty}(0, T; H)}  \leq C .
\end{align*}
Hence, there exists certain function $\bu \in L^2(0, T; V)$ which is a weak limit of $\{ \bu_n \}_n$. And just like before, we have also weak limits $\overline{\mathbb{S}(D\bu)}$ and $\overline{\bm s (\bu)}$. We observe that $\partial_t \bu_n$ are also naturally extended by 0 into $\Omega$.

To deal with the convective term, the strong convergence of $\bu_n$ in $L^2(L^2)$ will be needed. However, $\Omega$ is not a bounded set, so we cannot apply Aubin-Lions theorem directly as we do not have compactness. Nevertherles, thanks to \eqref{eq:UniformEstimates} and the simple fact that $V^*(Q_n) \subset V^*(Q_N)$ for all $n \geq N$, we get $|| \partial_t \bu_n ||_{L^2(0, T; V^*(Q_N))}  \leq C$. Together with the previous estimate and \Cref{thm:CrucialEstimates} we obtain that for all $n \geq N$ there holds
\begin{align*}
|| \bu_n ||_{L^2(0, T; W^{1,2}(Q_N))} +|| \partial_t \bu_n ||_{L^2(0, T; V^*(Q_N))}  \leq C .
\end{align*}
Now, thanks to Aubin-Lions theorem, we see that for every $N \in \mathbb{N}$ we have 
\begin{align*}
\bu_n \to \bu \text{ in } L^2(0, T; L^2(Q_N)) .
\end{align*}

It remains to verify that this limit $\bu$ indeed satisfies \eqref{eq:WF}. Due to density of $\mathcal{V}(Q_n)$ in $V$ it is enough to consider $\bm \varphi \in \mathcal{V}(Q_N)$ for some $N \in \mathbb{N}$. Such $\bm \varphi$ is of course in $V(Q_N)$ and as $V(Q_N) \subset V(Q_n) $, for $n \geq N$, and due to its compact support in $Q_N$, we see that \eqref{eq:WFApprox} gives us
\begin{align*}
\langle \partial_t \bu_n, \bm \varphi \rangle_{V(Q_N)} + \int\limits_{Q_N} \mathbb{S}(D\bu_n) : D \bm \varphi + \alpha \int\limits_{\Gamma_N} \bm s (\bu_n) \cdot \bm \varphi + \int\limits_{Q_N}( \bu_n \cdot \nabla ) \bu_n \cdot \bm \varphi = \langle (\bm f, \bm h), \bm \varphi \rangle _{V(Q_N)} \text{ for all } n \geq N .
\end{align*}
We multiply it by $\psi \in \mathcal{C}_c^{\infty}(0, T)$ and integrate over time interval $(0, T)$ and pass to the limit using our weak convergences and strong convergence in $ L^2(0, T; L^2(Q_N))$. Due to arbitrariness of $\psi$, density of $\mathcal{V}(Q_n)$ in $V$, and \eqref{eq:MonotoneNorms} we achieve
\begin{align}
\langle \partial_t \bu, \bm \varphi \rangle + \int\limits_{\Omega} \overline{ \mathbb{S}(D\bu) } : D \bm \varphi + \alpha \int\limits_{\Gamma} \overline{ \bm s (\bu) } \cdot \bm \varphi + \int\limits_{\Omega}( \bu \cdot \nabla ) \bu \cdot \bm \varphi = \langle (\bm f, \bm h), \bm \varphi \rangle \label{eq:WFApprox2}
\end{align}
for all $\bm \varphi \in V$. We note that $\partial_t \bu \in L^2(0, T; V^*)$ follows naturally and to conclude the proof of existence part, we need to identify weak limits $\overline{ \mathbb{S}(D\bu) }$ and $\overline{ \bm s (\bu) } $.

\underline{Step IV: weak limits of non-linearities.} In this part, we employ monotonicity arguments together with the well-known Minty's trick to show that  $\overline{ \mathbb{S}(D\bu) } = \mathbb{S}(D\bu)$ and $\overline{ \bm s (\bu) } =  \bm s (\bu)$. In contrast with the very similar procedure for $ m \to + \infty$, i.e. the identification of $\overline{ \mathbb{S}(D\bu_n) }$ and $\overline{ \bm s (\bu_n) } $, we need to be careful as we work with both $\Omega$ and $Q_n$ and we pass with $n \to + \infty$. First, we set $\bm \varphi = \bu$ in \eqref{eq:WFApprox2} to obtain 
\begin{align}
\int\limits_0^t \int\limits_{\Omega} \overline{\mathbb{S}(D\bu)} : D \bu + \alpha \int\limits_0^t \int\limits_{\Gamma} \overline{\bm s (\bu)} \cdot \bu =
\int\limits_0^t \langle (\bm f, \bm h), \bu \rangle  + \frac{1}{2} || \bu_0 ||_H^2 - \frac{1}{2} || \bu (t) ||_H^2  . \label{eq:equality2}
\end{align}

Now, as \eqref{eq:Coercivity} and \eqref{eq:CoercivityBoundary} hold, we have
\begin{align}
\begin{split}
c_1\int\limits_0^t || \bu_n ||_{V(Q_n)}^2 
&\leq 
\int\limits_0^t \int\limits_{Q_n} \mathbb{S}(D\bu_n) : D \bu_n + \alpha \int\limits_0^t \int\limits_{\Gamma_n} \bm s (\bu_n) \cdot \bu_n \\
&\leq 
\int\limits_0^{t+\varepsilon} \int\limits_{Q_n} \eta \, \mathbb{S}(D\bu_n) : D \bu_n  + \alpha \int\limits_0^{t+\varepsilon} \int\limits_{\Gamma_n} \eta \, \bm s (\bu_n) \cdot \bu_n    , \label{eq:inequality1}
\end{split}
\end{align}
where (for small $\varepsilon > 0$)
\begin{align*}
\eta (\tau) = 
\begin{cases}
1 \, & \text{ if } 0 \leq \tau \leq t \\ 1 + \frac{t-\tau}{\varepsilon} \, &  \text{ if } t < \tau < t + \varepsilon \\ 0 \, &  \text{ if } t + \varepsilon \leq \tau \leq T  
\end{cases} .
\end{align*}
We set $\bm \varphi = \bu_n$ in \eqref{eq:WFApprox}, multiply it by $\eta$ and integrate over $(0, T)$ to find
\begin{align*}
\int\limits_0^{t+\varepsilon} \int\limits_{Q_n} \eta \, \mathbb{S}(D\bu_n) : D \bu_n   + & \alpha \int\limits_0^{t+\varepsilon} \int\limits_{\Gamma_n} \eta \, \bm s (\bu_n) \cdot \bu_n   
=\int\limits_0^{t+\varepsilon} \eta \langle (\bm f, \bm h) , \bu_n\rangle_{V(Q_n)} - \frac{1}{2}\int\limits_0^{t+\varepsilon} \eta \cdot \frac{{\rm d}}{{\rm d}t} ||\bu_n ||_{H(Q_n)}^2
\\
&= 
\int\limits_0^{t+\varepsilon} \eta \langle (\bm f, \bm h) , \bu_n\rangle_{V(Q_n)}  + \frac{1}{2} || \bu_{0,n} ||_{H(Q_n)}^2  -\frac{1}{2\varepsilon} \int\limits_t^{t+\varepsilon} ||\bu_n||_{H(Q_n)}^2  \\
&= \int\limits_0^{t+\varepsilon} \eta \langle (\bm f, \bm h) , \bu_n\rangle  + \frac{1}{2} || \bu_{0,n} ||_{H}^2  -\frac{1}{2\varepsilon} \int\limits_t^{t+\varepsilon} ||\bu_n||_{H}^2   .
\end{align*}
Here we used that $\bu_n$ is extended by zero outside of $Q_n$. For $m \to +\infty$ we would obtain analogous relation as follows. We would multiply \eqref{eq:Galerkin} by $c_{n,k}^m$, sum it over $k$'s (the convective term vanishes again), multiply by $\eta (\tau)$ and integrate over the whole time inteval $(0, T)$. Instead of  $|| \bu_{0,n} ||_{H}^2$ we would deal with $ || P^m \bu_{0,n} ||_{H(Q_n)}^2$, but \Cref{thm:Basis} bounds it from above by  $|| \bu_{0,n} ||_{H(Q_n)}^2$.

Thanks to our weak convergences (in the whole space) together with the weak lower semicontinuity of the norm and \eqref{eq:MonotoneNorms}, we can now deduce from \eqref{eq:inequality1} that
\begin{align*}
\limsup_{n \to + \infty} \int\limits_0^t \int\limits_{Q_n}  \mathbb{S}(D\bu_n) : D \bu_n  +  \alpha \int\limits_0^t \int\limits_{\Gamma_n} \bm s (\bu_n) \cdot \bu_n  
&\leq  
\int\limits_0^{t+\varepsilon} \eta  \langle (\bm f, \bm h) , \bu \rangle + \frac{1}{2} ||  \bu_0 ||_{H}^2  -\frac{1}{2\varepsilon} \int\limits_t^{t+\varepsilon} ||\bu ||_{H}^2  .
\end{align*}
Using the weak lower semicontinuity again we can pass with $\varepsilon \to 0_+$ to get
\begin{align*}
\limsup_{n \to + \infty} \int\limits_0^t \int\limits_{Q_n}  \mathbb{S}(D\bu_n) : D \bu_n  +  \alpha \int\limits_0^t \int\limits_{\Gamma_n} \bm s (\bu_n) \cdot \bu_n  
&\leq  
\int\limits_0^t \eta  \langle (\bm f, \bm h) , \bu \rangle + \frac{1}{2} ||  \bu_0 ||_{H}^2  - \frac{1}{2} ||\bu (t) ||_{H}^2  
\end{align*}
and thanks to \eqref{eq:equality2} we achieve
\begin{align}
\limsup_{m \to + \infty} \int\limits_0^t \int\limits_{Q_n} & \mathbb{S}(D\bu_n) : D \bu_n +  \alpha \int\limits_0^t \int\limits_{\Gamma_n} \bm s (\bu_n) \cdot \bu_n  
\leq  
\int\limits_0^t \int\limits_{\Omega} \overline{\mathbb{S}(D\bu)} : D \bu + \alpha \int\limits_0^t \int\limits_{\Gamma} \overline{\bm s (\bu)} \cdot \bu  . \label{eq:inequality3}
\end{align}

We are now in the position to apply Minty's trick. We consider $\mathbb{V} \in L^2 (0, T; L^2(\Omega))$ and $\bv \in L^2 (0, T; L^2(\Gamma))$ be arbitrary. Due to \eqref{eq:Coercivity}, \eqref{eq:CoercivityBoundary}, and zero exstension of $\bu_n$ into $\Omega$, there holds
\begin{align*}
0 & \leq 
\int\limits_0^t \int\limits_{\Omega} (\mathbb{S} (D \bu_n) -  \mathbb{S} (\mathbb{V}) ) : (D \bu_n - \mathbb{V}) +   \alpha \int\limits_0^t \int\limits_{\Gamma} (\bm s (\bu_n) - \bm s (\bv) )\cdot (\bu_n - \bv ) \\
& = \int\limits_0^t \int\limits_{Q_n} \mathbb{S}(D\bu_n) : D \bu_n  +  \alpha \int\limits_0^t \int\limits_{\Gamma_n} \bm s (\bu_n) \cdot \bu_n  
- \int\limits_0^t \int\limits_{\Omega} \mathbb{S}(D\bu_n) : \mathbb{V}  -  \alpha \int\limits_0^t \int\limits_{\Gamma} \bm s (\bu_n) \cdot \bv \\
& \quad 
-\int\limits_0^t \int\limits_{\Omega}  \mathbb{S} (\mathbb{V})  : (D \bu_n - \mathbb{V}) -   \alpha \int\limits_0^t \int\limits_{\Gamma} \bm s (\bv) \cdot (\bu_n - \bv ) .
\end{align*}
We pass with $n \to + \infty$, using \eqref{eq:inequality3} and coupling the terms together we find
\begin{align*}
0 & \leq 
\int\limits_0^t \int\limits_{\Omega} (\overline{\mathbb{S}(D\bu)}  -  \mathbb{S} (\mathbb{V}) ) : (D \bu - \mathbb{V}) +   \alpha \int\limits_0^t \int\limits_{\Gamma} (  \overline{\bm s (\bu)} - \bm s (\bv) )\cdot (\bu - \bv ) .
\end{align*}
As $\mathbb{V}$ and $\bv$ were chosen arbitrary, we now set $\mathbb{V} = D\bu - \varepsilon \mathbb{W}$ and $\bv = \bu - \varepsilon \bm w$ for again arbitrary $\mathbb{W}$, $\bm w$ and $\varepsilon > 0$. Letting $\varepsilon \to 0_+$ we get 
\begin{align*}
0 & \leq 
\int\limits_0^t \int\limits_{\Omega} (\overline{\mathbb{S}(D\bu)}  -  \mathbb{S} (\mathbb{V}) ) :  \mathbb{W} +   \alpha \int\limits_0^t \int\limits_{\Gamma} (  \overline{\bm s (\bu)} - \bm s (\bv) )\cdot \bm w .
\end{align*}
Finally, thanks to the arbitrariness of both $\mathbb{W}$ and $\bm w$, we see that $\overline{\mathbb{S}(D\bu)} = \mathbb{S}(D\bu)$ and $ \overline{\bm s (\bu)} =  \bm s (\bu)$ indeed hold. So, the existence part of the proof is now complete as \eqref{eq:WF} is satisfied.

\begin{remark}
We note that the identification of the limit of the boundary non-linearity could be carried out more directly since $W^{1,2} (\Omega)$ is compactly embedded into $L^2(\partial \Omega)$ for a bounded 2D domain $\Omega$.
\end{remark} 
 
\underline{Step V: uniqueness.} Uniqueness of $\bu$ is just an application of the standard procedure. Nevertherles, the estimates there will be useful, so we briefly present it here. Let us consider two weak solutions $\bu$, $\bv$ to \eqref{eq:EquationInside}-\eqref{eq:Initial} with the same right-hand sides. For the difference $\bm w = \bu - \bv$, we have
\begin{align*}
\frac{1}{2} \cdot \frac{{\rm d}}{{\rm d}t} || \bm w ||_H^2 + c_1 || \bm w ||_V^2
\leq \Big| \int\limits_{\Omega} [ (\bu \cdot \nabla ) \bu - (\bv \cdot \nabla ) \bv ] \bm w \Big| ,
\end{align*}
where we used \eqref{eq:Coercivity} and \eqref{eq:CoercivityBoundary}. Due to zero divergence, H{\"o}lder's inequality, \eqref{eq:Ladyzhenskaya}, and \eqref{eq:SecondKorn} we find
\begin{align*}
\Big| \int\limits_{\Omega} [ (\bu \cdot \nabla ) \bu - (\bv \cdot \nabla ) \bv ] \cdot \bm w \Big| 
&= \Big| \int\limits_{\Omega}  (\bm w\cdot \nabla ) \bv \cdot \bm w \Big|
\leq || \bm w ||_{L^4(\Omega)}^2 || \nabla \bv ||_{L^2(\Omega)} \\
&\leq C || \bm w ||_{L^2(\Omega)} || \bm w ||_V  || \nabla \bv ||_{L^2(\Omega)}  \\
&\leq \varepsilon || \bm w ||_V^2 + C || \bm w ||_{L^2(\Omega)}^2 || \bv ||_{V}^2  .
\end{align*}
Putting it together and using Gr{\"o}nwall's inequality (recall that $\bv \in L^2(0, T; V)$) we get
\begin{align}
|| \bm w (t) ||_H^2 \leq C  || \bm w (0) ||_H^2 . \label{eq:ControlOfDifference1}
\end{align}
From this fact, uniqueness follows instantly. Moreover, integrating relative energy inequality above we find
\begin{align}
\int\limits_0^t || \bm w ||_V^2 \leq C  || \bm w (0) ||_H^2 . \label{eq:ControlOfDifference2}
\end{align}

\underline{Step VI: higher regularity.} Regarding the second part of the statement of \Cref{thm:ExistenceInfinitePlates}, we show the global version, i.e. for better initial condition. In details, for bounded domains, it can be found in Theorem 12 in \cite{PrZe24}; we mention that it comes down to Theorem III.3.5 in \cite{Te79}. 

To do so, we go back to \eqref{eq:Galerkin} and differentiate it in time. Multiply the result by $(c_{n, k}^m(t))'$, sum over $k$'s and after integration over time interval $(0, t)$ we obtain
\begin{align*}
\frac{1}{2}|| \partial_t \bu_n^m(t) ||_{H(Q_n)}^2 + c_3 \int\limits_0^t || \partial_t \bu_n^m||_{V(Q_n)}^2   & \leq \frac{1}{2}|| \partial_t \bu_n^m(0) ||_{H(Q_n)}^2 + \Big|  \int\limits_0^t \langle (\partial_t \bm f, \partial_t \bm h), \partial_t \bu_n^m\rangle_{V(Q_n)} \Big|  \\
& \quad +  \Big|  \int\limits_0^t \int\limits_{Q_n} [ ( \partial_t \bu_n^m \cdot \nabla)\bu_n^m + (\bu_n^m \cdot \nabla) (\partial_t \bu_n^m)  ] \cdot \partial_t \bu_n^m \Big|  ,
\end{align*}
where \eqref{eq:CoercivityDerivative} and \eqref{eq:CoercivityDerivativeBoundary} were used. Because $\bu_0$ is now more regular, it is possible to show that $\partial_t \bu_n^m(0) \in H$. The integral of duality can be splitted in a standard way using H{\"o}lder's and Young's inequalities. Finally, second part of the last integral vanishes as usual, and regarding the first one we have
\begin{align*}
\Big| \int\limits_0^t \int\limits_{Q_n} (\partial_t \bu_n^m \cdot \nabla)\bu_n^m  \cdot \partial_t \bu_n^m \Big| \leq \varepsilon || \partial_t \bu_n^m  ||_{V(Q_n)}^2 + C || \partial_t \bu_n^m  ||_{H(Q_n)}^2 || \bu_n^m  ||_{V(Q_n)}^2 ,
\end{align*}
just like in the estimation in Step V. Hence, \eqref{eq:UniformEstimates} holds also for $\partial_t \bu_n^m$ instead of $\bu_n^m$ and we have that 
\begin{align*}
|| \bu_n^m ||_{L^2(0, T; V(Q_n))} + || \partial_t \bu_n^m ||_{L^2(0, T; V(Q_n))} \leq C .
\end{align*}
Therefore, $|| \bu_n^m ||_{L^{\infty} (0, T; V(Q_n))} \leq C$ and $\bu_n$ are thus uniformly bounded in $L^{\infty} (0, T; V(Q_n))$. By extension it holds in the whole $\Omega$, and hence, the limit function satisfies $\bu \in L^{\infty} (0, T; V)$. Finally, higher regularity of $\partial_t \bu$ is clear as we have uniform estimate in $m$ and $\partial_t \bu_n$ are naturally extended by zero, which means the uniform estimates in $n$ also.
\end{proof}

The proof of the second main theorem would ideally follow using Theorem 1 from \cite{PrZe24} in $Q_n$ (which would be actually slightly modified to have $Q_n \in \mathcal{C}^{1,1}$). Nevertheless, it comes down to results of \cite{AACG21} and we do not know if constants therein do not explode with increasing $n$. Moreover, the proof from \cite{PrZe24} would still need to be modified as there were used embeddings that do not hold in unbounded domains. We thus have to do that from scratch, but the procedure remains the same. For just constructed unique weak solution in $\Omega$ we use the bootstrap argument relying on the following result. 

\begin{proposition}\label{thm:StationaryRegularity}
Let $\alpha > 0$, $p \in (1, +\infty )$ and consider
\begin{align*}
- \Delta \bu + \nabla \pi &=  \bm f \quad \text{ in }  \Omega  ,\\
 \alpha \bu + [(2D\bu) \bm n]_\tau &= \bm h \quad  \text{ on } \Gamma ,\\
\text{div}\, \bu  &= 0  \quad  \text{ in } \Omega ,\\
 \bu &= 0 \quad  \text{ on } \Gamma_1 , \\
 \bu \cdot \bm n &= 0 \quad  \text{ on } \Gamma . 
\end{align*}
If $\bm f \in L^{ t(p)} (\Omega),\, \bm h \in W^{-\frac{1}{p}, p}(\Gamma) \, \text{with } t(p) = \frac{ 2p }{p+2}$, then the unique solution $(\bu, \pi )$ of this system belongs to $W^{1, p}(\Omega) \times L^p(\Omega)$ and satisfies
\begin{align}
|| \bu ||_{W^{1,p}(\Omega)} + || \pi ||_{L^p(\Omega)} \leq C(\Omega, p, \alpha) \Big( || \bm f ||_{L^{ t(p)} (\Omega)} + || \bm h ||_{W^{-\frac{1}{p}, p}(\Gamma)} \Big) . \label{eq:StationaryRegularityFirst}
\end{align}
Also, if $\bm f \in L^p(\Omega),\, \bm h \in W^{1-\frac{1}{p}, p}(\Gamma)$, then this solution belongs to $W^{2, p}(\Omega) \times W^{1, p}(\Omega)$ and satisfies
\begin{align}
|| \bu ||_{W^{2,p}(\Omega)} + || \pi ||_{W^{1,p}(\Omega)} \leq C(\Omega, p, \alpha) \Big( || \bm f ||_{L^p(\Omega)} + || \bm h ||_{W^{1-\frac{1}{p}, p}(\Gamma)} \Big) . \label{eq:StationaryRegularitySecond}
\end{align}
\end{proposition}

\begin{proof}
It is a simple modification of Theorem 2.2 in \cite{PrZe25}. 
\end{proof}

This result shows that the boundary term $\beta \partial_t \bu$ in \eqref{eq:EquationBoundary} makes regularity substantially worse as $\bm h$ from \Cref{thm:StationaryRegularity} will include it during the bootstrap argument. In order to achieve $\bu \in L_{\text{loc}}^{\infty} (0, T; W^{2,p}(\Omega))$ we will thus need to show $\partial_t \bu \in L_{\text{loc}}^{\infty} (0, T; W^{1,p}(\Omega))$, because then $\text{tr}\, \partial_t \bu \in L_{\text{loc}}^{\infty} (0, T; W^{1-\frac{1}{p},p}(\Gamma))$. 

\begin{proof}[Proof of \Cref{thm:MaximalRegularity}]
We consider $\mathbb{S}(D\bu) = 2D\bu$ and $2 \leq p \leq 4$. We focus here only on the main points, see the section 3.2 in \cite{PrZe24} for details. From (local version of) \Cref{thm:ExistenceInfinitePlates} we have unique $\bu \in  L_{\text{loc}}^{\infty} (0, T; V)$, weak solution to \eqref{eq:EquationInside}-\eqref{eq:Initial}, with the time derivative in $L_{\text{loc}}^{\infty} (0, T; H) \cap L_{\text{loc}}^2 (0, T; V)$. The assertion of the theorem follows from \eqref{eq:StationaryRegularitySecond} if we show
\begin{align*}
\bm f - (\bu \cdot \nabla) \bu - \partial_t \bu &\in L_{\text{loc}}^{\infty} (0, T; L^p(\Omega)) ,\\
\beta \bm h + \alpha (\bu - \bm s (\bu))  - \beta \text{tr}\, \partial_t \bu &\in L_{\text{loc}}^{\infty} (0, T; W^{1-\frac{1}{p},p}(\Gamma))  .
\end{align*}

Due to $\bu \in  L_{\text{loc}}^{\infty} (0, T; V)$ we find $(\bu \cdot \nabla) \bu \in L_{\text{loc}}^{\infty} (0, T; L^{\frac{4}{3}} (\Omega) ) $. Observe that $W^{-\frac{1}{2}, 2} (\Omega) \hookrightarrow  L^{\frac{4}{3}} (\Omega)$ by Sobolev inequality. Using duality argument, it can be shown that the time derivative belongs to $L_{\text{loc}}^{\infty} (0, T; V^*)$ and as it is also in $ L_{\text{loc}}^{\infty} (0, T; L^2(\Omega))$, we see that $ \partial_t \bu \in L_{\text{loc}}^{\infty} (0, T; L^{\frac{4}{3}} (\Omega))$. Next, because of $\partial_t \bu \in L_{\text{loc}}^2 (0, T; V)$ we know that $\text{tr}\, \partial_t \bu$ makes sense and due to $\partial_t \bu \in L_{\text{loc}}^{\infty} (0, T; H)$ we have $\text{tr}\, \partial_t \bu \in L_{\text{loc}}^{\infty} (0, T; L^2(\Gamma))$. Thanks to Sobolev embedding we get $\text{tr}\, \partial_t \bu \in L_{\text{loc}}^{\infty} (0, T; W^{-\frac{1}{4}, 4}(\Gamma))$. Hence, $(\partial_t \bu, \text{tr}\, \partial_t \bu) \in L_{\text{loc}}^{\infty} (0, T; L^{\frac{4}{3}} (\Omega) \times W^{-\frac{1}{4}, 4}(\Gamma) ) $. Thanks to \eqref{eq:StationaryRegularityFirst} we now achieve $\bu \in L_{\text{loc}}^{\infty}(0, T; W^{1,4} (\Omega))$, in particular $\bu$ is bounded in $(0, T) \times \overline{\Omega}$. Using this information we can further improve the convective term, we get $(\bu \cdot \nabla) \bu \in L_{\text{loc}}^{\infty} (0, T; L^4 (\Omega) ) $. So, we already have the sufficient regularity of the convective term and our data $\bm f$, $\bm h$. Note that as $\partial_t \bu \in L_{\text{loc}}^2 (0, T; V)$ we can use \eqref{eq:StationaryRegularitySecond} to find $\bu \in L_{\text{loc}}^2 (0, T; W^{2,2}(\Omega))$.

It remains to improve the time derivative. Let us note that at this point we know that $\partial_t \bu$  lies just in $L_{\text{loc}}^{\tilde{q}} (0, T; L^{\tilde{p}}(\Omega))$ for some $2 < \tilde{q}$ and $2 < \tilde{p} < 4$. It can be shown that $\bv = \partial_t \bu$ satisfies the system
\begin{align}
\begin{split}
\partial_t \bv - \Delta \bv  +  \nabla \sigma &=  -(\bv \cdot \nabla)\bu  -(\bu \cdot \nabla)\bv  \text{ in } (0, T) \times \Omega ,   \\
\beta \partial_t \bv +  \alpha \bv + [(2D\bv) \bm n ]_\tau   &= \alpha ( 1- \nabla \bm s (\bu)) \bv  \text{ on } (0, T) \times \Gamma , \label{eq:SystemDerivative}
\end{split}
\end{align} 
together with divergence-free condition, impermeability and zero Dirichlet condition on $\Gamma_1$. 
\begin{remark}
As the right-hand side of \eqref{eq:SystemDerivative} belongs to $L_{\text{loc}}^2(0, T; H)$ we could try to use the approach from the end of the proof of Theorem 13 in \cite{PrZe24}. It is basically testing by $-\Delta \bv_n$ on the Galerkin level. However, one needs $L^2(Q_n) \hookrightarrow L^{\frac{4}{3}}(Q_n)$ which is not uniform in $n$ and it would not lead to $L_{\text{loc}}^{\infty}$-regularity in time anyway.

\end{remark}
The solution of system \eqref{eq:SystemDerivative} is unique and was obtained as the time derivative of the limit of $\bu_n$ in $Q_n$ and these are limits of the Galerkin approximations $\bu_n^m$. Note that $\partial_t \bv \in L_{\text{loc}}^2(0, T; V^*)$, which follows from the second part of \Cref{thm:ExistenceInfinitePlates}. We use the same strategy as in Step V from the proof of \Cref{thm:ExistenceInfinitePlates}, i.e. we differentiate the Galerkin approximation with respect to time and test by $\partial_t \bv_n^m$; here we use $\bm s \in \mathcal{C}^2$. All estimates are uniform and we find that $\partial_t \bv \in L_{\text{loc}}^{\infty} (0, T; H) \cap L_{\text{loc}}^2 (0, T; V)$. Note that thanks to $\bv \in L_{\text{loc}}^2 (0, T; V)$, we have $\bv = \partial_t \bu \in L_{\text{loc}}^{\infty} (0, T; V)$, so for $p = 2$ we already achieve $\bu \in L_{\text{loc}}^{\infty} (0, T; W^{2,2}(\Omega))$ using \eqref{eq:StationaryRegularitySecond}. For $2 < p \leq 4$, we proceed just like before and find $(\partial_t \bv, \text{tr}\, \partial_t \bv) \in L_{\text{loc}}^{\infty} (0, T; L^{\frac{4}{3}} (\Omega) \times W^{-\frac{1}{4}, 4}(\Gamma) ) $. Now, due to \eqref{eq:StationaryRegularityFirst} we find $\bv = \partial_t \bu \in L_{\text{loc}}^{\infty}(0, T; W^{1,4} (\Omega))$. As $\partial_t \bu \in L_{\text{loc}}^{\infty} (0, T; V)$, we get $\partial_t \bu \in L_{\text{loc}}^{\infty} (0, T; W^{1,p}(\Omega))$, which is the desired regularity of the time derivative and the proof is complete.

\end{proof}

\section{Attractor dimension}
This section purely focuses on the Navier-Stokes system, i.e. $\mathbb{S} (\mathbb{D}) = 2\nu \mathbb{D}$ for some $\nu > 0$. We suppose that \eqref{eq:CoercivityBoundary}-\eqref{eq:CoercivityDerivativeBoundary} hold (all constants there are thus equal to $2\nu$) and let $\nabla \bm s$ be Lipschitz. In the spirit of \Cref{thm:NonDimensionalization}, we can assume $\nu = 1$ and $L = 1$. Due to the first part of \Cref{thm:ExistenceInfinitePlates} it is possible to introduce a continuous semigroup $\{ \mathcal{S}(t) \}_{t \geq 0} : H \to H$ by 
\begin{align*}
\mathcal{S}(t) \bu_0 = \bu (t) \text{ for } t \geq 0 ,
\end{align*}
where $\bu_0 \in H$ and $\bu$ is the unique weak solution of \eqref{eq:EquationInside}-\eqref{eq:Initial}. Now, for time-independent $(\bm f, \bm h) \in H$, using \eqref{eq:CoercivityBoundary} relation \eqref{eq:EnergyEquality} reads
\begin{align}
\frac{1}{2} \cdot \frac{{\rm d}}{{\rm d}t} || \bu ||_H^2  + 2|| \bu ||_V^2  &= \langle  (\bm f, \bm h), \bu \rangle = ((\bm f, \bm h), \bu)_H   \label{eq:EnergyEqualitySpecific}
\end{align}
and estimate \eqref{eq:EquivalenceNorms} gives us (recall that $\Lambda =  \frac{32}{\pi^2} + \beta \min \big\{  \frac{1}{\alpha},  8 \big\}$)
\begin{align*}
\frac{{\rm d}}{{\rm d}t} || \bu ||_H^2   &\leq 2  ||\bu ||_H  \Big( || (\bm f, \bm h)||_H - \frac{2} {\Lambda} ||\bu||_H \Big) . 
\end{align*}
Hence the ball $B (0, R) \subset H $, with $R < \frac{\Lambda}{2}  || (\bm f, \bm h)||_H$, is uniformly absorbing. Further, it can be shown that  $\{ \mathcal{S}(t) \}_{t \geq 0}$ is asymptotically compact in $H$, i.e.  $\{ \mathcal{S}(t_n) \bu_n \}_{n \in \mathbb{N}}$ is precompact for any bounded $\{ \bu_n \}_{n \in \mathbb{N}}$ and $t_n \to + \infty$. We will not present the proof here, it relies on energy equality \eqref{eq:EnergyEquality} and can be done just like in section 3 in \cite{Rosa}. Therefore, solution semigroup $\{ \mathcal{S}(t) \}_{t \geq 0}$ possess a (unique) global attractor $\mathcal{A}$; see Theorem 3.1 therein. It means that $\mathcal{A} \subset H$ is a compact set, which is invariant with respect to $\mathcal{S}(t)$, i.e. $\mathcal{S}(t) \mathcal{A} = \mathcal{A} $ for all $t \geq 0$, and for any bounded $B \subset H$ there holds $\text{dist}\, (\mathcal{S}(t)B, \mathcal{A} ) \to 0$ as $t \to +\infty$, where we consider Hausdorff semi-distance, i.e. $\text{dist}\, (\mathcal{S}(t)B, \mathcal{A} ) = \sup_{a \in \mathcal{A}} \inf_{b \in \mathcal{S}(t)B} || b - a ||_H$. See \cite{Robinson11}, \cite{Robinson01} or \cite{Te97} for more details.

Our current goal is to estimate the fractal dimension of $\mathcal{A}$, i.e. to find an upper bound of
\begin{align*}
\dim_H^f  \mathcal{A} = \limsup\limits_{\varepsilon \to 0_+} \frac{\log N_{\varepsilon} \mathcal{A}}{- \log \varepsilon} \, ,
\end{align*}
where $N_{\varepsilon} \mathcal{A}$ denotes the minimal number of $\varepsilon$-balls (in $H$) needed to cover $\mathcal{A}$. To this aim we consider a formal linearization of our system \eqref{eq:EquationInside}-\eqref{eq:Initial}, i.e. 
\begin{align}
\partial_t \bU - \Delta \bU + (\bU \cdot \nabla ) \bu + (\bu \cdot \nabla ) \bU + \nabla \sigma &=    0 \text{ in } (0, T) \times \Omega , \label{eq:LinearInside}  \\
\beta \partial_t \bU + \alpha \nabla \bm s (\bu) \bU + [(2D\bU) \bm n ]_\tau   &=  0 \text{ on } (0, T) \times \Gamma ,  \label{eq:LinearBoundary}
\end{align}
together with
\begin{align}
\text{div}\, \bU &= 0 \text{ in } (0, T) \times \Omega ,  \\
\bU &= 0  \text{ on } (0, T) \times \Gamma_0 ,  \\
\bU \cdot \bm n &= 0 \text{ on } (0, T) \times ( \Gamma \cup \Gamma_0 ),  \\
\bU(0, \cdot) &= \bv_0 (\cdot) - \bu_0 (\cdot) \text{ in } \overline{ \Omega } .   \label{eq:LinearInitial}
\end{align}
Based on the previous arguments, it is simple to observe that it has a unique weak solution. It is useful to write these equations as 
\begin{align} 
\partial_t \bm U  = \mathcal{L}(t,\bu_0) \bm U ,  \label{eq:lin-rce}
\end{align}
where the solution operator $\mathcal{L}(t,\bu_0)$ depends on the solution $\bu=\bu(t)$ with $\bu(0) = \bu_0 \in \mathcal{A}$. We can now prove the following statement about relation between $\mathcal{S}(t)$ and $\mathcal{L}(t,\bu_0)$. It is the same assertion as Theorem 15 in \cite{PrZe24}, however, the proof there depends on $L^{\infty}(W^{2,p})$-regularity as in \Cref{thm:MaximalRegularity} and, as the following proof shows, it is not needed in the case of linear $\mathbb{S}$.

\begin{lemma}\label{thm:Differentiability}
The solution operator $\mathcal{L}(t,\bu_0)$ of \eqref{eq:LinearInside}--\eqref{eq:LinearInitial} is a uniform quasidifferential to $\mathcal{S}(t)$ on $\mathcal{A}$, i.e. for fixed $t \geq 0$ and any $\bu_0 \in \mathcal{A}$ there holds
\begin{align}
|| \mathcal{S}(t) \bv_0 - \mathcal{S}(t) \bu_0 - \bU(t) ||_H = o ( ||\bv_0 - \bu_0 ||_H ), \quad ||\bv_0 - \bu_0 ||_H \rightarrow 0 , \label{eq:QuasiDiff} 
\end{align}
where $\bv_0 \in \mathcal{A}$ and $\bU$ solves \eqref{eq:LinearInside}--\eqref{eq:LinearInitial}.
\end{lemma}

\begin{proof}
We consider $t > 0$, $\bu_0$, $\bv_0 \in \mathcal{A}$, and denote $\bu  = \mathcal{S}(t) \bu_0 $, $\bv  = \mathcal{S}(t) \bv_0 $. We start with substracting the equations for $\bm w = \bv - \bu $ and $\bU$ to obtain that 
\begin{align*}
\partial_t (\bm w - \bU) - \Delta (\bm w - \bU) +(\bv \cdot \nabla ) \bv - (\bu \cdot \nabla ) \bu - (\bU \cdot \nabla )  \bu - (\bu \cdot \nabla ) \bU + \nabla \pi - \nabla \sigma = 0 .
\end{align*}
Next, we test it by $\bm w - \bU$, which yields 
\begin{align}
\frac{1}{2} \cdot \frac{{\rm d}}{{\rm d}t} || \bm w - \bU||_H^2  + \int\limits_{\Omega} 2|D (\bm w - \bU)|^2 +  \alpha I = J, \label{eq:Rovnost}
\end{align}
where
\begin{align*}
I &= \int\limits_{\Gamma} [\bm s (\bv) - \bm s (\bu) - \nabla \bm s (\bu) \bU] \cdot (\bm w - \bU) , \\
J & = -\int\limits_{\Omega}  \left[ (\bv \cdot \nabla ) \bv -  (\bu \cdot \nabla ) \bu - (\bU \cdot \nabla ) \bu - ( \bu \cdot \nabla ) \bU  \right] \cdot (\bm w - \bU) .
\end{align*}

The first integral is rewritten using the mean value theorem (for some $\theta \in [0,1]$) as follows
\begin{align*}
I 
&= \int\limits_{\Gamma} [\nabla \bm s (\bu + \theta \bm w)\bm w  - \nabla \bm s (\bu) \bU \pm \nabla \bm s (\bu) \bm w] \cdot (\bm w - \bU) \\
&= \int\limits_{\Gamma} [\nabla \bm s (\bu + \theta \bm w)\bm w  - \nabla \bm s (\bu) \bm w] \cdot (\bm w - \bU) 
+ \int\limits_{\Gamma} \nabla \bm s (\bu) ( \bm w - \bU)  \cdot (\bm w - \bU) \\
&= I_1 + I_2.
\end{align*}
To estimate $I_1$ we recall that $\nabla \bm s$ is Lipschitz and for $I_2$ we invoke \eqref{eq:CoercivityDerivativeBoundary} to find
\begin{align*}
I_1 
& \leq C \int\limits_{\Gamma} |\bm w|^2 |\bm w - \bU|
\leq \varepsilon  \int\limits_{\Gamma}  |\bm w - \bU|^2 + C (\varepsilon) \int\limits_{\Gamma} |\bm w|^4 , \\
I_2 
& \geq 2 \int\limits_{\Gamma} |\bm w - \bU|^2 . 
\end{align*}

Next, the integral $J$ can be (see e.g. proof of Theorem 15 in \cite{PrZe24}) rewritten in the following form
\begin{align*}
J 
=  -\int\limits_{\Omega}    (\bm w \cdot \nabla ) \bm w  \cdot (\bm w - \bU) - \int\limits_{\Omega} [ (\bm w -\bU) \cdot \nabla ] \bu   \cdot (\bm w - \bU) .
\end{align*}
Now, in the first integral, we use integration by parts; Young's inequality in both integrals then gives us that
\begin{align*}
J &\leq \int\limits_{\Omega}  |\bm w|^2 | \nabla  (\bm w - \bU)| + \int\limits_{\Omega} |\bm w -\bU|^2 |\nabla \bu| \\
&\leq \frac{\varepsilon}{8+\frac{32L^2}{\pi^2}} \int\limits_{\Omega}  | \nabla  (\bm w - \bU)|^2 + C(\varepsilon) || \bm w ||_{L^4{({\Omega})}}^4  +  ||\bm w -\bU||_{L^4({\Omega})}^2 ||\nabla \bu||_{L^2({\Omega})}  \\
&\leq \varepsilon ||\bm w -\bU||_V^2 + C(\varepsilon) || \bm w ||_{L^4{({\Omega})}}^4  +  c  ||\bm w -\bU||_{L^2({\Omega})} ||\bm w -\bU||_V  ||\nabla \bu||_{L^2({\Omega})}  \\
&\leq 2\varepsilon ||\bm w -\bU||_V^2 + C(\varepsilon) || \bm w ||_{L^4{({\Omega})}}^4  + C(\varepsilon) ||\bm w -\bU||_{L^2({\Omega})}^2 ||\nabla \bu||_{L^2({\Omega})}^2 ,
\end{align*}
where we also used \eqref{eq:Korn} and \eqref{eq:Ladyzhenskaya}. Now, \eqref{eq:Rovnost}, together with the previous estimates, leads to
\begin{align*}
 \frac{{\rm d}}{{\rm d}t} || \bm w - \bU||_H^2  + 2|| \bm w - \bU ||_V^2 
\leq  C \big[ || \bm w ||_{L^4{({\Omega})}}^4 + || \bm w ||_{L^4{({\Gamma})}}^4 \big]   +  C  ||\bm w -\bU||_{L^2({\Omega})}^2 ||\nabla \bu||_{L^2({\Omega})}^2 ,
\end{align*}
and due to Grönwall's inequality, we obtain
\begin{align*}
|| (\bm w - \bU)(t)||_H^2  
& \leq 
 C  \int\limits_0^t \Big( || \bm w ||_{L^4{({\Omega})}}^4 + || \bm w ||_{L^4{({\Gamma})}}^4 \Big) \exp \Big( C\int\limits_0^t ||\nabla \bu||_{L^2({\Omega})}^2 \Big) \\
 & \leq  Ce^{CT}  \int\limits_0^t \Big( || \bm w ||_{L^4{({\Omega})}}^4 + || \bm w ||_{L^4{({\Gamma})}}^4 \Big) ,
\end{align*}
where we used $\nabla \bu \in L^\infty (0, T; L^2({\Omega}))$ since $\bu_0 \in \mathcal{A}$. The integral over $\Omega$ in the expression above can be estimated as follows
\begin{align*}
 \int\limits_0^t ||\bm w||_{L^4({\Omega})}^4 \leq  C \int\limits_0^t ||\bm w||_{L^2({\Omega})}^2 ||\bm w||_V^2 \leq C ||\bm w_0||_H^2  \int\limits_0^t  ||\bm w||_V^2 \leq C ||\bm w_0||_H^4 ,
\end{align*}
where we used interpolation \eqref{eq:Ladyzhenskaya} and estimates \eqref{eq:ControlOfDifference1}, \eqref{eq:ControlOfDifference2}. For the boundary integral we first interpolate $L^4$ between $L^2$ and $L^{\infty}$, and then use the observation that $W^{1,2}(\Omega) \hookrightarrow L^{\infty}(\Gamma)$ (as $\Omega$ is two-dimensional). We have
\begin{align*}
 \int\limits_0^t ||\bm w||_{L^4({\Gamma})}^4 \leq  \int\limits_0^t ||\bm w||_{L^2({\Gamma})}^2 ||\bm w||_{L^{\infty}(\Gamma)}^2 \leq  C \int\limits_0^t ||\bm w||_H^2 ||\bm w||_{W^{1,2}(\Omega)}^2 \leq C ||\bm w_0||_H^4 ,
\end{align*}
where the last inequality is obtained as before. The conclusion \eqref{eq:QuasiDiff} immediately follows.
\end{proof}

Now, we introduce the so-called $N$-trace of the linearized equation:
\begin{align*}	
q(N) = \limsup_{t\to +\infty} \sup_{\bu_0 \in \mathcal{A}} \sup_{ \{\bm \varphi_j \}_{j=1}^N} \frac{1}{t} \int\limits_0^t	\sum_{j=1}^N (\mathcal{L}(\tau,\bu_0)\bm \varphi_j,\bm \varphi_j)_H \, {\rm d}\tau ,
\end{align*}
where the last supremum is taken over all families of functions $\{\bm \varphi_j \}_{j=1}^N$ from $V$, which are orthonormal in $H$. Due to quasidifferentiability of $\mathcal{S}(t)$, the quantity $q(N)$ provides an effective tool to find an upper bound on the attractor dimension. This is the subject of the following result.

\begin{proposition}\label{thm:Dimension}
Suppose that $q(N) \leq f(N)$, where $f(\cdot)$ is a concave function, and $f(d)=0$ for some $d>0$. Then $\dim_H^f \mathcal{A} \leq d$.
\end{proposition}

\begin{proof}
See the comment at the beginning of the second section in \cite{ChepIl04} about the non-compact case and Corollary 3.1 therein.
\end{proof}

Before we start estimating the expression above, we rewrite it first as follows
\begin{align}	
q(N) = \limsup_{t\to +\infty} \sup_{\bu_0 \in \mathcal{A}} \sup_{ \{ \bm \varphi_j \}_{j=1}^N} \frac{2}{t} \int\limits_0^t	\sum_{j=1}^N (\mathcal{L}(\tau,\bu_0) \bm \varphi_j, \bm \varphi_j)_H \, {\rm d}\tau , \label{eq:N-trace}
\end{align}
where we now take $ \bm \varphi_j \subset V$ with $( \bm \varphi_j , \bm \varphi_i)_H = \frac{1}{2} \delta_{ij}$. The reason why we introduce rescaled $\bm \varphi_j$ is the following. As in \Cref{thm:BadKorn}, we can extend these functions into $\mathbb{R} \times (-1, 1)$ and denote them as $E\bm \varphi_j$. Then, all these $E\bm \varphi_j$ belong into $W^{1,2}_0 (\mathbb{R} \times (-1, 1))$ and every finite family $\{ E\bm \varphi_j \}_{j=1}^N$ is the so-called suborthonormal in $L^2(\mathbb{R} \times (-1, 1))$, which means that 
\begin{align*}
\sum_{i,j=1}^N \xi_i \xi_j (E\bm \varphi_i, E\bm \varphi_j)_{L^2(\mathbb{R} \times (-1, 1))}
& \leq \sum_{i=1}^N  \xi_i^2 
\end{align*}
holds for all $\{\xi_j \}_{j=1}^N \subset \mathbb{R}$, see e.g. \cite{GMT88}. It indeed holds as the ensuing calculation shows:
\begin{align*}
\sum_{i,j=1}^N \xi_i \xi_j (E\bm \varphi_i, E\bm \varphi_j)_{L^2(\mathbb{R} \times (-1, 1))}
& = \sum_{i,j=1}^N \xi_i \xi_j \int\limits_{\mathbb{R}} \int\limits_{-1}^1  E\bm \varphi_i \cdot E\bm \varphi_j 
= \sum_{i,j=1}^N \xi_i \xi_j \int\limits_{\Omega} 2\bm \varphi_i \cdot \bm \varphi_j \\
& = \sum_{i,j=1}^N \xi_i \xi_j \big( \delta_{ij} - 2\beta(\bm \varphi_i,\bm \varphi_j)_{L^2(\Gamma)} \big) 
 = \sum_{i=1}^N \xi_i^2 -  2\beta \Big|\Big| \sum_{i=1}^N  \xi_i \bm \varphi_i \Big|\Big|_{L^2(\Gamma)}^2 \\
& \leq \sum_{i=1}^N  \xi_i^2 .
\end{align*}
This property will be useful in a while thanks to the following Lieb-Thirring inequality.
\begin{proposition} \label{thm:Lieb-Thirring}
There exists a constant $\kappa$ such that for every open $\Omega \subset \mathbb{R}^2$ and every finite family $\{ \bm \varphi_j \}_{j=1}^N \subset W_0^{1,2}(\Omega)$ which is suborthonormal in $L^2 (\Omega)$ we have
 \begin{align*}
\Big|\Big| \sum_{j=1}^N |\bm \varphi_j|^2 \Big|\Big|_{L^2(\Omega)}^2 \leq \kappa  \sum_{j=1}^N  ||\nabla \bm \varphi_j ||_{L^2(\Omega)}^2  .
\end{align*}
\end{proposition}

\begin{proof}
Follows directly from Corollary 4.3 in \cite{GMT88}, with $m=1$ and $n=k=p=2$.
\end{proof}

Note that we can take $\kappa = \frac{1}{2\sqrt{3}}$ in the standard Dirichlet setting, where the orthonormality is considered, see \cite{IPZ2016}. We now have all the ingredients to complete the proof of the last result of our paper.

\begin{proof}[Proof of \Cref{thm:DimensionEstimate}]
We suppose that $\bm s (\bu) = 2 \bu$ (as $\nu = 1$) and the goal is to estimate \eqref{eq:N-trace}. There holds
\begin{equation*}
- (\mathcal{L}(\cdot,\bu_0) \bm \varphi_j, \bm \varphi_j)_H 
=2 ||  \bm \varphi_j ||_V^2 - \int\limits_{\Omega} ( \bm \varphi_j \cdot \nabla) \bu \cdot  \bm \varphi_j - (\bu \cdot \nabla ) \bm \varphi_j \cdot  \bm \varphi_j ,
\end{equation*}
and thus
\begin{equation*}	
\sum_{j=1}^N  (\mathcal{L}(\cdot,\bu_0) \bm \varphi_j, \bm \varphi_j)_H
\leq -  2\sum_{j=1}^N ||  \bm \varphi_j ||_V^2 + || \nabla \bu ||_{L^2(\Omega)} \Big|\Big| \sum_{j=1}^N |\bm \varphi_j|^2 \Big|\Big|_{L^2(\Omega)} .
\end{equation*}
Invoking now \Cref{thm:Lieb-Thirring} and \Cref{thm:BadKorn} we can estimate the second term as
\begin{align*}
|| \nabla \bu ||_{L^2(\Omega)} \Big|\Big| & \sum_{j=1}^N |\bm \varphi_j|^2 \Big|\Big|_{L^2(\Omega)}
\leq || \nabla \bu ||_{L^2(\Omega)} \Big|\Big| \sum_{j=1}^N |E \bm \varphi_j|^2 \Big|\Big|_{L^2(\mathbb{R}\times (-1, 1))} \\
& \leq  \frac{1}{8} \sum_{j=1}^N || \nabla (E \bm \varphi_j) ||_{L^2 (\mathbb{R} \times (-1, 1))}^2  + 2\kappa  || \nabla \bu ||_{L^2(\Omega)}^2 \\ 	
& \leq   \sum_{j=1}^N || D \bm \varphi_j ||_{L^2 (\Omega)}^2  + 2\kappa  || \nabla \bu ||_{L^2(\Omega)}^2  	
\leq   \sum_{j=1}^N || \bm \varphi_j ||_V^2 + 2\kappa  || \nabla \bu ||_{L^2(\Omega)}^2 .
\end{align*}
Together, using also \eqref{eq:EquivalenceNorms}, \eqref{eq:FirstKorn}, and $( \bm \varphi_j , \bm \varphi_i)_H = \frac{1}{2} \delta_{i, j}$, we obtain
\begin{align*}
\sum_{j=1}^N  (\mathcal{L}(\cdot,\bu_0)\bm \varphi_j,\bm \varphi_j)_H
&\leq -  \sum_{j=1}^{N} || \bm \varphi_j ||_V^2 + 2\kappa  || \nabla \bu ||_{L^2(\Omega)}^2 
\leq - \frac{1}{\Lambda} \sum_{j=1}^{N} || \bm \varphi_j ||_H^2 + 16\kappa || D \bu ||_{L^2(\Omega)}^2 \\
&\leq - \frac{1}{2\Lambda}  \sum_{j=1}^{N} 1 + 16\kappa || D \bu ||_{L^2(\Omega)}^2 
 \leq - \frac{N}{2\Lambda} + 16\kappa  || \bu ||_V^2 .
\end{align*}	
Observe now that integrating \eqref{eq:EnergyEqualitySpecific} over the time interval gives us 
\begin{align*}
2\int\limits_0^t || \bu ||_V^2  \leq t R || (\bm f, \bm h)||_H + \frac{1}{2} R^2 < \frac{t}{2}\Lambda  || (\bm f, \bm h)||_H^2 + \frac{1}{2} R^2. 
\end{align*}
Substituting all these into \eqref{eq:N-trace} we achieve 
\begin{align*}	
q(N) 
& \leq  - \frac{N}{\Lambda} + 8\kappa \Lambda || (\bm f, \bm h)||_H^2 .
\end{align*}
So, \Cref{thm:Dimension} shows that $\dim_H^f \mathcal{A} \leq 8 \kappa \Lambda^2 || (\bm f, \bm h)||_H^2$, where $\Lambda =  \frac{32}{\pi^2} + \beta \min \big\{  \frac{1}{\alpha},  8 \big\}$. This estimate was obtained for rescaled variables in the sense of \Cref{thm:NonDimensionalization}. Hence, using the last argument therein and formulas for $\alpha^*$, $\beta^*$ we find
\begin{align*}
\dim_{H_L}^f \mathcal{A} 
&\leq 8 \kappa \Big[ \frac{32}{\pi^2} + \beta^* \min \Big\{  \frac{1}{\alpha^*}  , 8 \Big\} \Big]^2 \frac{L^4}{\nu^4} || (\bm f, \bm h)||_{H_L}^2 
=  \frac{8 \kappa}{\nu^4} \Big[ \frac{32L^2}{\pi^2} + \beta \min \Big\{  \frac{1}{\alpha} , 8 L \Big\} \Big]^2 || (\bm f, \bm h)||_{H_L}^2  ,
\end{align*}
which is exactly \eqref{eq:DimensionEstimate} and the proof is complete.

\end{proof}

\begin{center}
Acknowledgement
\end{center}
I wish to thank Dalibor Pra\v{z}\'{a}k for fruitful discussions and valuable suggestions during this work.

\bibliographystyle{amsplain}
\bibliography{bibliography}

\providecommand{\bysame}{\leavevmode\hbox to3em{\hrulefill}\thinspace}
\providecommand{\MR}{\relax\ifhmode\unskip\space\fi MR }
\providecommand{\MRhref}[2]{%
  \href{http://www.ams.org/mathscinet-getitem?mr=#1}{#2}
}
\providecommand{\href}[2]{#2}
\begin{thebibliography}{10}

\bibitem{ABM2021}
A.~Abbatiello, M.~Bul\'{\i}\v{c}ek, and E.~Maringov\'{a}, \emph{On the dynamic slip boundary condition for {N}avier--{S}tokes-like problems}, Mathematical Models and Methods in Applied Sciences \textbf{31} (2021), no.~11, 2165--2212.

\bibitem{Abergel}
F.~Abergel, \emph{Existence and finite dimensionality of the global attractor for evolution equations on unbounded domains}, Journal of Differential Equations \textbf{83} (1990), no.~1, 85--108.

\bibitem{AACG21}
P.~Acevedo, C.~Amrouche, C.~Conca, and A.~Ghosh, \emph{Stokes and {N}avier-{S}tokes equations with {N}avier boundary conditions}, J. Differential Equations \textbf{285} (2021), 258--320. \MR{4231512}

\bibitem{BMR2007}
M.~Bul\'{\i}\v{c}ek, J.~M\'{a}lek, and K.~R. Rajagopal, \emph{Navier's slip and evolutionary {N}avier-{S}tokes-like systems with pressure and shear-rate dependent viscosity}, Indiana Univ. Math. J. \textbf{56} (2007), no.~1, 51--85. \MR{2305930}

\bibitem{ChepIl01}
V.~V. Chepyzhov and A.~A. Ilyin, \emph{A note on the fractal dimension of attractors of dissipative dynamical systems}, Nonlinear Anal., Theory Methods Appl., Ser. A, Theory Methods \textbf{44} (2001), no.~6, 811--819 (English).

\bibitem{ChepIl04}
\bysame, \emph{On the fractal dimension of invariant sets: Applications to {N}avier-{S}tokes equations}, Discrete and Continuous Dynamical Systems \textbf{10} (2004), no.~1\&2, 117--135.

\bibitem{CFT88}
P.~Constantin, C.~Foias, and R.~Temam, \emph{On the dimension of the attractors in two-dimensional turbulence}, Physica D: Nonlinear Phenomena \textbf{30} (1988), no.~3, 284--296.

\bibitem{GMT88}
J.-M. Ghidaglia, M.~Marion, and R.~Temam, \emph{Generalization of the {S}obolev-{L}ieb-{T}hirring inequalities and applications to the dimension of attractors}, Differential Integral Equations \textbf{1} (1988), no.~1, 1--21. \MR{920485}

\bibitem{Hatz12}
S.~G. Hatzikiriakos, \emph{Wall slip of molten polymers}, Progress in Polymer Science \textbf{37} (2012), 624--643.

\bibitem{IPZ2016}
A.~Ilyin, K.~Patni, and S.~Zelik, \emph{Upper bounds for the attractor dimension of damped {N}avier-{S}tokes equations in {$\Bbb{R}^2$}}, Discrete Contin. Dyn. Syst. \textbf{36} (2016), no.~4, 2085--2102. \MR{3411554}

\bibitem{Liu93}
V.~X. Liu, \emph{A sharp lower bound for the {H}ausdorff dimension of the global attractors of the 2{D} {N}avier-{S}tokes equations}, Communications in Mathematical Physics \textbf{158} (1993), 327--339.

\bibitem{EM-dis}
E.~Maringov\'{a}, \emph{Mathematical analysis of models arising in continuum mechanics with implicitly given rheology and boundary conditions}, Ph.D. thesis, Faculty of Mathematics and Physics, Charles University, Prague, 2019.

\bibitem{PrPr23}
D.~Pra\v{z}ák and B.~Priyasad, \emph{The existence and dimension of the attractor for a 3{D} flow of a non-{N}ewtonian fluid subject to dynamic boundary conditions}, Applicable Analysis \textbf{103} (2023), no.~1, 166--183.

\bibitem{PrZe25}
D.~Pra\v{z}ák and M.~Zelina, \emph{On {L}$^p$ semigroup to {S}tokes equation with dynamic boundary condition in the half-space}, arXiv preprint arXiv:2312.04478 (2023).

\bibitem{PrZe24x2}
\bysame, \emph{On the uniqueness of the solution and finite-dimensional attractors for the 3{D} flow with dynamic slip boundary condition}, Differential Integral Equations \textbf{37} (2024), no.~11-12, 859--880.

\bibitem{PrZe24}
\bysame, \emph{Strong solutions and attractor dimension for 2{D} {NSE} with dynamic boundary conditions}, Journal of Evolution Equations \textbf{24} (2024), no.~20.

\bibitem{Robinson01}
J.~C. Robinson, \emph{Infinite-dimensional dynamical systems}, Cambridge University Press, 2001.

\bibitem{Robinson11}
\bysame, \emph{Dimensions, embeddings, and attractors}, Cambridge University Press, 2011.

\bibitem{RRS2016}
J.~C. Robinson, J.~L. Rodrigo, and W.~Sadowski, \emph{Three-{D}imensional {N}avier-{S}tokes {E}quations}, Cambridge University Press, 2016.

\bibitem{Rosa}
R.~Rosa, \emph{The global attractor for the 2d {N}avier-{S}tokes flow on some unbounded domains}, Nonlinear Analysis: Theory, Methods \& Applications. \textbf{32} (1998), no.~1, 71--85.

\bibitem{Te79}
R.~Temam, \emph{Navier-{S}tokes {E}quations: {T}heory and {N}umerical {A}nalysis}, North-Holland, 1979.

\bibitem{Te97}
\bysame, \emph{Infinite-dimensional dynamical systems in mechanics and physics}, second ed., Applied Mathematical Sciences, vol.~68, Springer-Verlag, New York, 1997. \MR{MR1441312 (98b:58056)}

\end{thebibliography}

\end{document}